\documentclass[11pt]{amsart}
\usepackage{amsmath,amssymb,amsfonts,amsthm, amscd,indentfirst}
\usepackage{amsmath,latexsym,amssymb,amsmath,%awheremscd,
amscd,amsthm,amsxtra}
\usepackage{hyperref}\usepackage{url}
\usepackage{color}
\usepackage{pdfpages}
\usepackage{graphics}
\usepackage{enumitem}

\usepackage{graphicx}
\usepackage{caption}

\usepackage{epsfig,here}
\usepackage{subfigure,here}
\DeclareMathOperator{\arccosh}{arccosh}
\newtheorem{defi}{Definition}[section]
\newtheorem{theorem}{Theorem}[section]
\newtheorem{prop}{Proposition}[section]

\newtheorem{corollary}{Corollary}[section]

\newtheorem{remark}{\textbf{Remark}}[section]

\def\rr{\mathbb{R}}
\def\ss{\mathbb{S}}
\def\hh{\mathbb{H}}
\def\bb{\mathbb{B}}

\def\tr{\mathrm{tr}}

\def\p{\partial}

\def\a{\alpha}

\def\d{\delta}
\def\th{\theta}
\def\l{\lambda}

\def\p{\partial}

\def\<{\langle}
\def\>{\rangle}
\def\div{{\rm div}}
\def\n{\nabla}

\def\De{\Delta}
\def\vp{\varphi}
\def\R{{\mathbb R}}
\def\ep{\epsilon}

\numberwithin{equation} {section}

\begin{document}

\title[Stability for a second type partitioning problem]{Stability for a second type partitioning problem}
\author{Jinyu Guo}
\address{School of Mathematical Sciences\\
Xiamen University\\
361005, Xiamen, P.R. China}
\email{guojinyu14@163.com}
\author{Chao Xia}
\address{School of Mathematical Sciences\\
Xiamen University\\
361005, Xiamen, P.R. China}
\email{chaoxia@xmu.edu.cn}
\thanks{This work is  supported by NSFC (Grant No. 11871406),  the Natural Science Foundation of Fujian Province of China (Grant No. 2017J06003) and the Fundamental Research Funds for the Central Universities (Grant No. 20720180009).
}

\begin{abstract}
In this paper, we study stability and instability problem for type-II partitioning problem. First, we make a complete classification of stable type-II stationary hypersurfaces in a ball in a space form as totally geodesic $n$-balls. Second, for general ambient spaces and convex domains, we give some topological restriction for type-II stable stationary immersed surfaces in two dimension. Third, we give a lower bound for the Morse index for type-II stationary hypersurfaces in terms of their topology.\\


\end{abstract}

\date{}
\maketitle

\medskip

\tableofcontents

%{\bf Keywords:} Capillary surfaces, Stability, Rigidity, Surface with free boundary.
\section{Introduction}
Let $B\subset \rr^{n+1}$ be a convex body (compact convex set with non-empty interior). We look at hypersurfaces in $B$ which divides $B$ into to two disjoint domains $B_1$ and $B_2$ in different manners.
In the literature the following two types of partitioning problems have been considered.

\smallskip

\noindent{\bf Type-I partitioning problem.} Find the area-minimizing hypersurfaces among all hypersurfaces in $B$ which divides $B$ into two disjoint domains $B_1$ and $B_2$ with prescribed volume, namely, $$|B_1|=s|B|\hbox{ and }|B_2|=(1-s)|B|, \hbox{ for some }s\in (0,1).$$

\noindent {\bf Type-II partitioning problem.} Find the area-minimizing hypersurfaces among all such hypersurfaces in $B$ which divides $B$ into two disjoint domains $B_1$ and $B_2$ with prescribed wetting boundary area\footnote{wetting boundary means for the boundary part on $\p B$. The word ``wetting" comes from the physical model of capillary surfaces.}, namely,
$$|B_1\cap \p B|=s|\p B|\hbox{ and }|B_2\cap \p B|=(1-s)|\p B|, \hbox{ for some }s\in (0,1).$$

\smallskip

These two problems haven been intensively studied by Burago-Maz'ya in late 60s \cite{BM}  in the case $B=\bar \bb^{n+1}$, the unit ball. By using spherical symmetrization, they showed that the solution for Type-I partitioning problem is totally geodesic $n$-ball and all spherical caps intersecting $\ss^n (=\p \bb^{n+1})$ orthogonally (see \cite{Ma}, Section 5.2.1, Lemma 1), while the solution for Type-II partitioning problem is all totally geodesic $n$-balls  (see \cite{Ma}, Section 9.4.4, Lemma). Bokowsky-Sperner \cite{BS} also studied these two partitioning problems and gave the same classification result as Burago-Maz'ya when $B$ is a ball. They also gave several estimates for the corresponding isoperimetric ratio when $B$ is a general convex body.

Besides the area-minimizing hypersurfaces, one is also interested in studying stationary hypersurfaces for these partitioning problems. It follows from the first variational formulas (see e.g. Section 2) that stationary hypersurfaces for Type-I partitioning problem are free boundary constant mean curvature (CMC)  hypersurfaces, while stationary hypersurfaces for Type-II partitioning problem are minimal hypersurfaces intersecting $\p B$ at a constant angle. Here free boundary means the hypersurfaces intersects $\p B$ orthogonally.
There have been plenty of works about existence, regularity and construction of  free boundary CMC hypersurfaces, especially free boundary minimal hypersurfaces in the last four decades.

When $B$ is a ball, there are several rigidity results, for example, Hopf type theorem by Nitsche \cite{Ni} and Ros-Souam \cite{RS}, Alexandrov type theorem by Ros-Souam \cite{RS} for these two types stationary hypersurfaces. In particular, it has been found by Fraser-Schoen's series of works \cite{FS1, FS2, FS3} that  free boundary minimal hypersurfaces in a ball turns out to have close relationship with the Steklov eigenvalue problem.

Stability problem for Type-I partitioning problem has been initiated and studied by Ros-Vergasta \cite{RV}.
A free boundary CMC hypersurface is called type-I stable if the second variation of the area functional at this hypersurface is non-negative among any volume-preserving variations.
Type-I stable free boundary CMC hypersurfaces are smooth local minimizers for Type-I partitioning problem. In the framework of sets of finite perimeter, the local minimizers for Type-I partitioning problem have been considered by Sternberg-Zumbrun \cite{SZ} in 1998.
It has been conjectured that the free boundary totally geodesic $n$-ball and free boundary spherical caps are all type-I stable free boundary CMC hypersurfaces in $\bar \bb^{n+1}$. This conjecture has been recently solved by Nunes \cite{Nu} in two dimension (see also Barbosa \cite{Ba}) and Wang-Xia \cite{WX} in any dimensions. Moreover, Wang-Xia \cite{WX} also gave complete classification for type-I stable capillary hypersurfaces in $\bar \bb^{n+1}$, namely, CMC hypersurfaces intersecting $\ss^n$ at a constant contact angle.

The first aim of this paper is to study stability problem for type-II partitioning problem. Recall that stationary hypersurfaces for type-II partitioning problem in $B$ are minimal hypersurfaces intersecting $\p B$ at a constant angle. We call a minimal hypersurfaces intersecting $\p B$ at a constant angle is type-II stable if the second variation of the area functional at this hypersurface is non-negative among any wetting-area-preserving variations.
We show the following complete classification for type-II stable minimal hypersurfaces in $\bar \bb^{n+1}$ intersecting $\ss^n$ at a constant angle.
\begin{theorem}\label{thm0.1}
A type-II stable stationary immersed hypersurface in an Euclidean ball $\bar \bb^{n+1}$ is a totally geodesic $n$-ball.
\end{theorem}
In case that the ball lies in a space form, we have the similar result.
\begin{theorem}\label{thm0.2}
A type-II stable stationary immersed hypersurface in a (n+1)-ball in a space form is a totally geodesic $n$-ball.
\end{theorem}

For general convex bodies in general ambient $3$-manifolds, we obtain some topological restriction for  type-II stable stationary immersed surfaces.
\begin{theorem}\label{thm0.3}
Let $x: M\to B\subset \bar M^3$ be a type-II stable stationary compact immersed surface with free boundary. Assume $\overline{\rm Ric}\ge 0$ and $h^{\p B}\ge 0$.
Then the only possible values for the genus $g$ and the number of boundary component $r$ of $x(M)$ are
$g=0\hbox{ or }1$ and $r=1, 2 \hbox{ or }3$ or $g=2$ and $r=1$. Moreover, $g=2$ and $r=1$ happens only when $h^{\p B}\equiv 0$ along $\p M$ and $\overline{{\rm R}}-\overline{{\rm Ric}}(\nu,\nu)\equiv 0$ along $M$.
\end{theorem}

\begin{theorem}\label{thm0.4}
Let $x: M\to B\subset \bar M^{3}$ be a type-II stable stationary compact immersed surface. Assume $\overline{\rm Sect}\ge 0$ and $h^{\p B}\ge 0$. Assume $x(\p M)$ is embedded in $\p B$.
Then the only possible values for the genus $g$ and the number of boundary component $r$ of $x(M)$ are
$g+\frac{r}{2}<4$ if $g$ is even and $g+\frac{r}{2}<5$ if $g$ is odd.
\end{theorem}

In Theorems \ref{thm0.3} and \ref{thm0.4}, $\overline{\rm Sect}$, $\overline{\rm Ric}$ and $\overline{{\rm R}}$ denote the sectional curvature, the Ricci curvature and the scalar curvature of $\bar M$ respectively, and $h^{\p B}$ denotes the second fundamental form of $\p B\subset \bar M$.

\

The instability for a variational problem is quantitatively measured by the Morse index. For the type-II partitioning problem, the Morse index for a stationary hypersurface is the non-negative integer which indicates the dimension of sets of wetting-area-preserved deformations which decreases the area of the type-II stationary hypersurface. A stationary hypersurface is stable is equivalent that it has vanishing Morse index. It turns out that the Morse index  controls the topology and geometry for stationary hypersurfaces.

There are plenty of works on the index estimate for closed minimal hypersurfaces or minimal hypersurfaces with free boundary, see for example, Ros \cite{Ros2}, Savo \cite{Sa} and Ambrozio-Carlotto-Sharp \cite{ACS1, ACS2}.
See also \cite{Ros1, CdO1, CdO2} for index estimate for CMC surfaces with free boundary, which is related to type-I partitioning problem. The technique in \cite{Ros1, CdO1, CdO2} for non-minimal CMC case only applies for two dimension.

The next aim of this paper is to study the index estimate for minimal hypersurfaces with constant contact angle, i.e., stationary hypersurfaces for type-II partitioning problem. We use ${\rm Ind}(M)$ to denote the Morse index for a type-II stationary hypersurface $M$. Following the argument of Savo \cite{Sa} and Ambrozio-Carlotto-Sharp \cite{ACS1, ACS2}, by using the coordinates of harmonic one-forms, we are able to prove the following lower bound for the index.
\begin{theorem}\label{thm-6}
Let $x: M^n\to B\subset \bar M^{n+1}$ be a type-II stationary compact immersed hypersurface. Let $\bar M$ be isometrically embedded in $\mathbb{R}^d$. Assume for any non-zero vector field $\xi$ on $M$ satisfies
\begin{eqnarray}\label{condi-6-5}
&&\int_{M} {\rm tr}_M(\overline{{\rm Rm}}(\cdot,\xi,\cdot,\xi))  +\overline{{\rm Ric}}(\nu,\nu)|\xi|^2dA+\int_{\p M} \frac{1}{\sin \th}\, H^{\p B}|\xi|^{2} ds\\
&>&\int_{M}\left[|({\rm II}(\cdot, \xi)|^2-|{\rm II}(\nu, \xi)|^2)+(|{\rm II}(\cdot, \nu)|^2- |{\rm II}(\nu, \nu)|^2)|\xi|^2 \right]\,dA.\nonumber
\end{eqnarray}
where $\overline{{\rm Rm}}$ and $\overline{{\rm Ric}}$ denote the Riemannian curvature tensor and Ricci curvature tensor of $\bar M$ respectively, $H^{\p B}$ denotes the mean curvature of $\p B\subset \bar M$, and ${\rm II}$ denotes the second fundamental form for the embedding $\bar M\subset \rr^d$.
Then
\begin{equation*}
{\rm Ind}(M)\geq\frac{2}{d(d-1)}\dim H_{1}(M, \partial M; \mathbb{R})-1
\end{equation*}
where $ H_{1}(M,\partial M;\mathbb{R})$ denotes the first relative homology group with real coefficients.
\end{theorem}
As a corollary, we have the following
\begin{corollary}\label{cor-7}
Let $B$ be a strictly mean convex domain in $\rr^{n+1}$ and $x: M^n\to B\subset \rr^{n+1}$ be a type-II stationary compact immersed hypersurface. Then
\begin{equation*}
{\rm Ind}(M)\geq\frac{2}{n(n+1)}\dim H_{1}(M,\partial M;\mathbb{R})-1.
\end{equation*}
\end{corollary}

Using another argument of Ros \cite{Ros1} and Ambrozio-Carlotto-Sharp \cite{ACS1}, again using the coordinates of harmonic one-forms, we can get the following Morse index estimate for two dimension.
\begin{theorem}\label{thm-7-p}
Let $x: M^2\to B\subset \bar M^{3}$ be a type-II stationary compact immersed surface. Let $\bar M$ be isometrically embedded in $\mathbb{R}^d$. Assume for any non-zero vector field $\xi$ on $M$ satisfies
\begin{eqnarray}\label{condi-7}
&&\int_{M} |{\rm II}(\cdot,\xi)|^{2}-|{\rm II}(\nu,\xi)|^{2}-\frac{1}{2}\overline{{\rm R}}|\xi|^{2}dA-\int_{\partial M}\frac{1}{\sin\theta}H^{\partial B}|\xi|^{2}ds<0
\end{eqnarray}
Then
\begin{equation*}
{\rm Ind}(M)\geq\frac{1}{d}(2g+r-1)-1.
\end{equation*}
\end{theorem}

In the case of $\bar{M}^{3}=\mathbb{R}^{3} \hbox{ or }\mathbb{S}^{3}$ and $B$ is mean convex domain,  the inequality (\ref{condi-7}) is satisfied. (In the case $\mathbb{R}^{3}$, one needs use the fact that there are no minimal closed surfaces in $\mathbb{R}^{3}$).
% If the surface is compact and orientable, the estimate must involve the genus of the surface by \cite{ACS2} lemma 5.
Therefore we have the following corollaries.
\begin{corollary}\label{dim-2-R}
Let $B$ be a mean convex domain in $\rr^3$ and $x: M^2\to B$ be a type-II stationary compact immersed surface. Then
\begin{equation*}
{\rm Ind}(M)\geq\frac{1}{3}(2g+r-4).
\end{equation*}
\end{corollary}
\begin{corollary}\label{dim=2-S}
Let $B$ be a mean convex domain in $\ss^3$ and $x: M^2\to B$ be a type-II stationary compact immersed surface.
Then
\begin{equation*}
  {\rm Ind}(M)\geq\frac{1}{4}(2g+r-5).
\end{equation*}
\end{corollary}

The remaining part of this paper is organized as follows. In Section \ref{sec2} we review the definition and basic properties of type-II stationary hypersurfaces. %Since we are concerned with the immersions,a suitable notion of area and the so-called wetting area  functional is needed to study type-II hypersurfaces.
In Section \ref{sec3} we give a proof of Theorem \ref{thm0.1} for type-II stationary hypersurfaces in a ball in $\R^{n+1}$ after finding admissible test function \eqref{test-function1}.
and we will provide a detailed proof of Theorem \ref{thm0.2} for type-II stationary   hypersurfaces in a ball in ${\mathbb H}^{n+1}$ and sketch
a proof for type-II hypersurfaces in a ball in ${\mathbb S}^{n+1}$.
In Section \ref{sec4},  we use balancing argument to study stability problem for type-II stationary   hypersurfaces in general convex bodies in general ambient manifolds, and prove Theorem \ref{thm0.3} and \ref{thm0.4}.
In Section \ref{sec5}, we give Morse index estimate lower bounds for type-II stationary hypersurfaces, and prove Theorems \ref{thm-6} and \ref{thm-7-p},

\

\section{Preliminaries}\label{sec2}
Let $(\bar M^{n+1}, \bar g)$ be an oriented $(n+1)$-dimensional Riemannian manifold and $B$ be a smooth
compact domain in $\bar M$ that is diffeomorphic to an Euclidean ball. Let $x: (M^n, g)\to B$ be an
isometric immersion of an orientable $n$-dimensional compact manifold $M$ with boundary $\p M$ into $B$  satisfying $$x({\rm int} M)\subset {\rm int} B\hbox{ and }x(\p M)\subset\p B.$$
 Such an immersion is called {\it proper}.

We denote by $\bar \n$, $\bar \Delta$ and $\bar \n^2$ the gradient, the Laplacian and the Hessian on $\bar M$ respectively, while by $\n$, $\Delta$ and $\n^2$ the gradient, the Laplacian and the Hessian on $M$ respectively.
We will use the following terminology for four normal vector fields.
We choose one of the unit normal vector field along $x$ and denote it by $\nu$.
We denote  by $\bar N$ the unit outward normal to $\p B$ in $B$ and $\mu$ be the unit outward normal to $\p M$ in $M$.
Let $\bar \nu$ be the unit normal to $\p M$ in $\p B$ such that the bases $\{\nu, \mu\}$ and $\{\bar \nu, \bar N\}$ have the same orientation in the
normal bundle of $\p M\subset \bar M$.
%See Figure 1. %\ref{fig1}.
Denote by $h$ and $H$ the second fundamental form and the mean curvature of the immersion $x$ respectively. Precisely,
$h(X, Y)= \bar g(\bar \n_X \nu, Y)$ and $H=\tr_g(h).$
%For constant mean curvature hypersurfaces which are our main concern, we always choose $\nu$ to be one of the unit normal vector fields so that $H\ge 0$.

By an admissible variation of $x$, we mean a differentiable map $x: (-\epsilon, \epsilon)\times M\to B\subset \bar M$ such that $x(t, \cdot): M\to B$ is an immersion satisfying $x(t, {\rm int} M)\subset{\rm int} B$ and $x(t, \p M)\subset\p B$ for every $t\in (-\ep, \ep)$ and $x(0, \cdot)=x$.
We denote the area functional $A: (-\epsilon, \epsilon)\rightarrow \mathbb{R}$ of the immersion $x(\cdot, t)$ by
\begin{equation*}
  A(t)=\int_{M}dA_{t}
\end{equation*}
where $dA_{t}$ is the volume element of $M$ in the metric induced by $x(t, \cdot)$. We denote the wetting area functional $A_W(t): (-\epsilon, \epsilon)\rightarrow \mathbb{R}$ are defined by
\begin{equation*}
    A_{W}(t)=\int_{[0,t]\times\partial M}x^{\ast}dA_{\partial B}.
\end{equation*}
where $dA_{\partial B}$ is the area element of $\partial B$.
A variation is said to be wetting-area-preserving if $A_W(t)=A_W(0)=0$ for each $t\in (-\ep, \ep)$.

It is easy to see that the first variation formulae of $A(t)$ and $A_W(t)$ for an admissible variation with a variation vector field $Y=\frac{\p}{\p t}x(t,\cdot)|_{t=0}$ are given by
\begin{eqnarray}
&&A'(0)=\int_M  H\bar g(Y, \nu) \,dA+\int_{\p M} \bar g(Y, \mu)\,ds,\label{first-var}\\
&&A_W'(0) =\int_{\partial M}\bar g(Y, \bar{\nu})\, ds,\label{wet-first-var}
\end{eqnarray}
where $dA$ and $ds$ are the area element of $M$ and $\p M$ respectively.

%Variational vector field $Y=\frac{\partial}{\partial t}x(t,\cdot)|_{t=0}=Y_{0}+f\nu$. where $Y_{0}$ is tangent to $M$ and $\nu$ is unit normal vector for $M$ and $f=\langle Y,\nu\rangle$.\newline
%We denote the mean curvature of $x$ by $H$. For an arbitrary variation, it can be shown that
%\begin{alignat}{2}
%A'(0)&=\frac{d}{dt}\Big|_{t=0}(\int_{M}dA_{t})=\int_{M}(H\langle Y,\nu \rangle+div Y_{0})dA\label{first-var} \\
%&=\int_{M}HfdA+\int_{\partial M}\langle Y_{0},\mu\rangle ds\nonumber\\
%&=\int_{M}HfdA+\int_{\partial M}\langle Y,\mu\rangle \nonumber ds
%\end{alignat}
%where $dA$ and $ds$ are the ares element of $M$ and $\partial M$ respectively.

%Let $T$ be the domain which is enclosed by $\partial M$ in $\partial B$, by admissible conditions, we know that $Y$ is tangent to unit sphere $\mathbb{S}^{n}$, therefore we can view $Y$ as variational vector field of $\partial M$ with respect to $\mathbb{S}^{n}$. Hence preserving area of $T$ is equivalent to following formula \cite{BCE}
%\begin{equation}\label{area-preserving}
% V'_{\partial M}(0)=\int_{\partial M}\langle Y, \bar{v}\rangle ds=0
%\end{equation}
\begin{defi}
A proper immersion $x: M\rightarrow B\subset\bar M$ is said to be {\it type-II stationary} if $A'(0)=0$ for any wetting-area-preserving variation of $x$.
\end{defi}
From the above first variation formulae, we know that $x$ is type-II stationary if and only if $x$ is a minimal immersion, namely $H=0$, and \begin{eqnarray}\label{capill}
\bar g(Y, \mu-c\bar \nu)=0 \hbox{ for some constant }c\in\rr\hbox{ and any }Y\in T({\p B}).
\end{eqnarray}
Equation \eqref{capill} implies $\partial M$ intersects $\partial B$ at some constant angle $\theta\in (0,\pi)$ such that $\cos \th=c$.

%We make a convention on the choice of $\nu$ to be the opposite direction of mean curvature vector so that the mean curvature $H$ is always non-negative.
We make a choice of the normal $\nu$ so that, along $x(\partial M)$, the angle between $-\nu$ and $\bar{N}$ or equivalently between $\mu$ and $\bar{\nu}$ is everywhere equal to $\theta$ (see Figure 1).
To be more precise, in the normal bundle of $x(\partial M)$, we have the following relations:
\begin{alignat}{2}
\mu&=\sin\theta\, \bar{N}+\cos\theta\,\, \bar{\nu},\label{mu}\\
\nu&=-\cos\theta\,\, \bar{N}+\sin\theta\, \bar{\nu}.\label{nu}
\end{alignat}

\begin{figure}
\centering
\includegraphics[height=7cm,width=14.8cm]{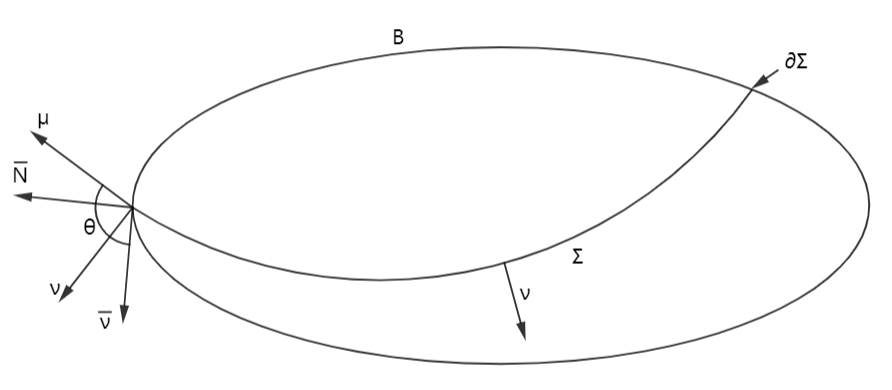}
\caption{$\Sigma=x(M)$ and $\partial\Sigma=x(\partial M)$}
\end{figure}

For $Y\in T(\p B)$, from \eqref{nu} and the fact $\bar g(Y,\bar N)=0$, we see that a wetting-area-preserving variation with $Y=Y_{0}+\vp\nu$, where $Y_0$ is the tangential part of $Y$, satisfies
\begin{equation}\label{wetting-pre}
  \int_{\partial M}\vp \,ds=0.
\end{equation}
Conversely, we shall show
\begin{prop}Let $x: M\rightarrow B\subseteq \bar{M}$ be a proper type-II stationary immersion with a contact angle $\th$. % intersecting $\partial B$ at an angle $\theta$.
Then for a given $\vp\in C^\infty(M)$ satisfying $\int_{\p M}  \vp \,ds=0$, there exists an admissible wetting-area-preserving variation of $x$ with variational vector field having $\vp\nu$ as its normal part.
\end{prop}

\begin{proof} We argue as in \cite{BCE} and \cite{AS}.

We first assume $x: M\to B$ is embedded.
For each point $p\in \p M$, let $\nu_0=\nu+\cos \th \, \bar N$ be the projection of $\nu$ on $T_{x(p)}(\p B)$. Denote $W=\frac{1}{\bar g(\nu,\nu_0)}\nu_0-\nu$
which is tangential to $x(M)$ along $\p M$. Extend $W$ smoothly to a vector field on $x(M)$, still denote by $W$. Denote $Z=W+\nu$ and extend $Z$ smoothly to a vector field on $U\subset B$, which is a $\delta$-neighborhood of $x(M)$ in $B$, such that $Z$ is tangential to $T(\p B)$ along $\p B\cap \bar U$.
By construction, $\bar g(Z, \nu)=1$. Consider the local flow $\psi_t$ of $Z$ in $\bar U$, that is, $\frac{\p}{\p t}\psi_t= Z$. Let
$\Psi: (-\ep_1, \ep_1)\times M\to B$ be given by $\Psi(t, \cdot)=\psi_t$.
We shall find a function $u:(-\ep,\ep)\times M\to \rr$ such that $$\tilde{\Psi}(t, \cdot)=\Psi(u(t,\cdot), \cdot)$$ is the desired deformation.
First, since $\psi_t$ is the local flow of $Z$ and $Z$ is  tangential to $T(\p B)$ along $\p B\cap \bar U$, we know $\tilde{\Psi}(t, \p M)\subset \p B$.
Second,
since $$\tilde{\Psi}^{\ast}dA_{\partial B}=\frac{\p u}{\p t} {\Psi}^{\ast}dA_{\partial B}=\frac{\p u}{\p t} E(u(t,\cdot), \cdot) dtdA_{\p M},$$
where $E(u(t,\cdot), \cdot)= \det(d\Psi|_{(u(t,\cdot), \cdot)}),$
we have $$A_W(\tilde{\Psi}(t, \cdot))=\int_{[0,t]\times\partial M}\tilde{\Psi}^{\ast}dA_{\partial B}= \int_{\p M}\int_0^t \frac{\p u}{\p t} E(u(t,\cdot), \cdot) dtdA_{\p M}.$$
Let $u(t, \cdot): (-\ep,\ep)\times M\to \rr$  be the local solution of the following initial value problem:
$$\frac{\p u}{\p t}= \frac{\vp}{E(u(t,\cdot), \cdot)}, \quad u(0,\cdot)=0, \hbox{ in }M.$$
It follows from the condition $\int_{\p M}  \vp \,ds=0$ that
$A_W(\tilde{\Psi}(t, \cdot))=0,$ that is, $\tilde{\Psi}(t, \cdot)$ is a wetting area preserving admissible deformation.
Finally, it is easy to see that
$$\frac{\p}{\p t}\Big|_{t=0}\tilde{\Psi}(t, \cdot)= \frac{\p u}{\p t}\Big|_{t=0}\cdot Z(0, \cdot)= \vp(W+\nu),$$
which means the variational vector field of $\tilde{\Psi}(t, \cdot)$ has $\vp\nu$ as its normal part. 

In the general case of immersion, we shall first construct an admissible variation $\tilde{x}: (-\epsilon, \epsilon)\times M\to B$ and endow $(-\epsilon, \epsilon)\times M$ with the pull-back metric $\tilde{x}^*(\bar g)$ and it is enough to prove the result for  $(-\epsilon, \epsilon)\times M$ endowed with $\tilde{x}^*(\bar g)$, which is the embedded case. See \cite{AS}, Proposition 2.1 for details.
\end{proof}

The second variational formula of the area functional $A$ under admissible wetting-area-preserving variations is given as follows.
\begin{prop}\label{second-var}
Let $x: M\rightarrow B\subset\bar M$ be a proper type-II stationary immersion. Let $x(\cdot, t)$ be an admissible wetting-area-preserving variation with variational vector field $Y$ having $\vp\nu$ as its normal part. Then
\begin{eqnarray}\label{second}
&&A''(0)=\int_M -\vp(\De \vp+(|h|^2+\overline{{\rm Ric}}(\nu,\nu))\vp)\, dA+\int_{\p M} \vp(\n_\mu \vp-q \vp)\,ds.
\end{eqnarray}
Here
\begin{eqnarray}\label{q}
q=\frac{1}{\sin \th}h^{\p B}(\bar \nu, \bar \nu)+\cot \th \, h(\mu,\mu),
\end{eqnarray}
$\overline{\rm Ric}$ is the Ricci curvature tensor of $\bar M$,  %$h$ is the second fundamental form of the immersion $x$  given by $h(X, Y)= \bar g(\bar \n_X \nu, Y)$
and $h^{\p B}$ is that of $\p B$ in $\bar M$ given by $h^{\p B}(X, Y)= \bar g(\bar \n_X \bar N, Y)$.
\end{prop}
We postpone the proof of Proposition \ref{second-var} to Appendix \ref{appendix}.
\begin{defi}
A proper type-II stationary immersion $x: M\rightarrow B\subset\bar M$ is called stable if $A''(0)\ge 0$ for all wetting-area-preserving variations, that is, \begin{eqnarray}\label{stab-ineq0}
\int_M -\vp(\De \vp+(|h|^2+\overline{{\rm Ric}}(\nu,\nu))\vp) dA+\int_{\p M} \vp(\n_\mu \vp-q \vp)\,ds\ge 0,\end{eqnarray}
for any $\vp$ satisfying $\int_{\p M} \vp\, ds=0$. Here $q$ is as in \eqref{q}.

\end{defi}

\
%Since $\langle Y, \bar{N}\rangle\mid_{\mathbb{S}^{n}}=0$ and $\theta$=const. Using (\ref{mu}) and (\ref{nu}), we get

%\begin{alignat}{2}
%\int_{\partial M} fds=\int_{\partial M}\langle Y,\nu\rangle ds&=\int_{\partial M}\langle Y,-\cos\theta\,\ \bar{N}+sin\theta\ \bar{v} \rangle ds
%&=sin\theta\int_{\partial M}\langle Y, \bar{v}\rangle ds=0.
%\end{alignat}
%Therefore, area of $T$ is preserved if and only if $f$ satisfies $\int_{\partial M}fds=0$.

\section{Uniqueness for type-II stable stationary hypersurfaces in a ball}\label{sec3}
\subsection{The Euclidean case}\

In this subsection, we consider the case $(\bar M, \bar g)=(\rr^{n+1}, \delta)$ and $B=\bar \bb^{n+1}$ is the Euclidean unit ball (in our notation, $\bb^{n+1}$ is the Euclidean unit open ball).
In this case, $\overline{\rm Ric}\equiv 0$, $h^{\p\bb}=g^{\p\bb}$ and $\bar N(x)=x$.
Abuse of notation, we use $x$ to denote the position vector in $\rr^{n+1}$. We use $\<\cdot, \cdot\>$ to denote the Euclidean inner product.

The stability condition becomes
\begin{eqnarray}\label{stab-ineq1}
A''(0)=\int_M -\vp(\De \vp+|h|^2\vp)\, dA+\int_{\p M} \vp(\n_\mu \vp-q \vp)\,ds\ge 0
\end{eqnarray}
with $$q=\frac{1}{\sin \th}+\cot \th \, h(\mu,\mu).$$
for all $\vp$ such that $\int_{\p M} \vp\, ds=0$.

\begin{theorem}
A type-II stable minimal hypersurfaces in $\bar \bb^{n+1}$ intersecting $\ss^n$ at a constant angle is a totally geodesic $n$-ball.
\end{theorem}
\begin{proof}
For convenience, we omit writing the volume form and the area form in an integral.\\
We know that $M$ is minimal and $\partial M$ intersects $\ss^n$ at a constant angle, say $\th\in (0, \pi)$.
For each constant vector field $a\in \rr^{n+1}$,
Define on $M$,
\begin{equation}\label{test-function1}
  \vp_a=\frac{1}{\sin \th}\,\<x, a\>+\cot\th \, \<\nu, a\>.
\end{equation}

%On $\p \S$, it holds
% \begin{eqnarray}
%&&\mu=\sin \th \, \bar N+\cos \th\, \bar \nu,  \label{mu0}\\&&\nu=-\cos \th \, \bar N+\sin \th\, \bar \nu.\label{nu0}
%\end{eqnarray}
By direct computation, by using \eqref{mu} and \eqref{nu}, one sees
\begin{eqnarray*}
\vp_a|_{\p M}&=&\frac{1}{\sin \th}\,\<\bar N, a\>+\cot\th (-\cos \th \, \<\bar N, a\>+\sin \th\, \<\bar \nu, a\>)
\\&=&\sin \th \, \<\bar N, a\>+\cos \th\, \<\bar \nu, a\>=\<\mu, a\>.
\end{eqnarray*}

On the other hand, since $M$ is minimal, we have $$\De \<x, a\>=0.$$ Thus, by integration by parts, we see
$$\int_{\p M}\vp_a= \int_{\p M}\<\mu, a\>= \int_{\p M}\n_{\mu} \<x, a\>= \int_{M}\De \<x, a\>=0.$$
Therefore, $\vp_a$ is an admissible test function in \eqref{stab-ineq1}.
It follows that
\begin{eqnarray}\label{stab-ineq2}
\int_M -\vp_a(\De \vp_a+|h|^2\vp_a) dA+\int_{\p M} \vp_a(\n_\mu \vp_a-q \vp_a)ds\ge 0.
\end{eqnarray}

We can compute that on $\p M$ by \cite{WX} Proposition 2.1,
\begin{eqnarray}\label{xeq1}
&&\n_\mu \vp_a= \frac{1}{\sin \th}\, \n_\mu \<x, a\>+\cot\th \, \n_\mu \<\nu,a\>
=\frac{1}{\sin \th}\,\<\mu, a\>+\cot\th \, h(\mu, \mu) \<\mu, a\>
=q\<\mu, a\>=q\vp_a.
\end{eqnarray}
Also, on $M$,
\begin{eqnarray*}
&&\De \vp_a= \frac{1}{\sin \th}\,\De \<x, a\>+\cot\th \,  \De \<\nu,a\>= -\cot\th\,|h|^2\<\nu,a\>.
\end{eqnarray*}
Thus
\begin{eqnarray}\label{xeq2}
&&\De \vp_a+|h|^2\vp_a=  -\cot\th\,|h|^2\<\nu,a\>+ |h|^2\left(\frac{1}{\sin \th}\,\<x, a\>+\cot\th \, \<\nu,a\>\right)=\frac{1}{\sin \th}\,|h|^2\<x, a\>.
\end{eqnarray}
Using \eqref{xeq1} and \eqref{xeq2} in \eqref{stab-ineq2}, we get that for each $a\in\rr^{n+1}$,
\begin{eqnarray}\label{xeq3}
\int_M \left(\frac{1}{\sin \th}\,\<x, a\>+\cot\th \, \<\nu,a\>\right) \frac{1}{\sin \th}\,|h|^2\<x, a\>\le 0.
\end{eqnarray}

We take $a$ to be the $n+1$ coordinate vectors $\{E_i\}_{i=1}^{n+1}$ in $\R^{n+1}$, and add \eqref{xeq3} for all $a$ to get
\begin{eqnarray}\label{xeq4}
&&\int_M \,|h|^2\left(|x|^2+\cos\th \, \<x,\nu\>\right) \le 0.
\end{eqnarray}
Let $\Phi=\<x,\nu\>+\cos\th$, we know that $\Phi|_{\p M}=0$ from \eqref{nu}. Thus $$\int_M\De \frac12\Phi^2= \int_{\p M}\Phi\n_\mu \Phi =0,$$
and combining with \eqref{xeq4}, we have
\begin{eqnarray}\label{xeq5}
&&\int_M \,|h|^2\left(|x|^2+\cos\th \, \<x,\nu\>\right)+\De  \frac12\Phi^2\le 0.
\end{eqnarray}
On the other hand, we have
$$\De \Phi=\De\<x,\nu\>=-|h|^2\<x,\nu\>.$$
It follows that
$$\De \frac12\Phi^2=\Phi\De\Phi+ |\n \Phi|^2= -|h|^2\<x,\nu\>^2-|h|^2\cos\th\,\<x,\nu\>+|\n \Phi|^2.$$
Thus, inequality \eqref{xeq5} reduces to
%\begin{eqnarray*}
%\int_M |h|^2\left(|x|^2+\cos\th\,\<x,\nu\>\right)+\De \frac12\Phi^2\le 0.
%\end{eqnarray*}
%which gives
\begin{eqnarray*}
\int_M |h|^2|x^T|^2+|\n \Phi|^2\le 0.
\end{eqnarray*}
We conclude that $h\equiv 0$ which implies that $M$ is a totally geodesic $n$-ball in $\bar \bb^{n+1}$.
\end{proof}

\smallskip

%\section{Type-II hypersurface in a ball in space forms}\label{sec4}
\subsection{The hyperbolic case}\

\noindent Let $\hh^{n+1}$ be the simply connected hyperbolic space with curvature $-1$. We use here the  Poincar\'e ball model, which  is given by
\begin{eqnarray}\label{Poincare}
\hh^{n+1}=\left(\bb^{n+1}, \bar g=e^{2u} \delta\right), \quad e^{2u}=\frac{4}{(1-|x|^2)^2}.
\end{eqnarray}

In this subsection we use $\delta$ or $\<\cdot, \cdot\>$ to denote the Euclidean metric %in $\bb^{n+1}\subset \rr^{n+1}$.
and the Cartesian coordinate in $\bb^{n+1}\subset \rr^{n+1}$.
Sometimes we also represent the hyperbolic metric, in terms of the polar coordinate with respect to the origin, as $$\bar g=dr^2+\sinh^2 rg_{\ss^n}.$$
We use $r=r(x)$ to denote the hyperbolic distance from the origin and denote $V_0=\cosh r$.
It is easy to verify that  \begin{eqnarray}\label{hypfunc}
V_0=\cosh r=\frac{1+|x|^2}{1-|x|^2},  \quad \sinh r=\frac{2|x|}{1-|x|^2}.
\end{eqnarray}
 The position function $x$, in terms of polar coordinate, can be represented by \begin{eqnarray}\label{rad-conf0}
x=\sinh r \p_r.
\end{eqnarray}
It is well-known that $x$ is a conformal Killing vector field with \begin{eqnarray}\label{rad-conf}
\bar\n x=V_0 \bar g.
\end{eqnarray}

Let $B^\hh_{R}$ be a ball in $\hh^{n+1}$ with hyperbolic radius ${R}\in (0, \infty)$. By an isometry of $\hh^{n+1}$, we may assume $B^{\hh}_{R}$ is centered at the origin.
$B^\hh_{R}$, when viewed as a set in $\bb^{n+1}\subset \rr^{n+1}$,  is the Euclidean ball with radius $R_{\rr}:=\sqrt{\frac{1-\arccosh R}{1+\arccosh R}} \in (0, 1)$.
The principal curvatures of $\p B^\hh_{R}$ are $\coth R$. The unit normal $\bar N$ to $\p B^{\hh}_R$ with respect to $\bar g$ is given by
 \begin{eqnarray}\label{NN}
\bar N=\frac{1}{\sinh R}x.
\end{eqnarray}
Moreover, for each constant vector field $a\in \mathbb{R}^{n+1}$, we can define a smooth vector field $Y_a$ in $\hh^{n+1}$ by
\begin{eqnarray}\label{KillY}
Y_a=\frac12(|x|^2+1)a-  \<x, a\>x.
\end{eqnarray}
From \cite{WX} Proposition 4.1, we know that $Y_a$ is a Killing vector field in $\hh^{n+1}$, i.e.,
\begin{eqnarray}\label{KillY1}
\frac12(\bar \n_i (Y_a)_j+ \bar \n_j (Y_a)_i)=0.
\end{eqnarray}
Denote function $V_{a}$ as follow
\begin{equation*}
  V_{a}= \frac{2\<x, a\>}{1-|x|^2}.
\end{equation*}
\begin{prop}\label{prop-WX}\cite{WX}
For any tangential vector field $Z$ on $\hh^{n+1}$, we have
\begin{eqnarray}
&&\bar \n_{Z}V_a=\bar g(Z, e^{-u}a)+ e^{-u}\bar g(x, e^{-u}a)\bar g(Z, x),\label{eq3h}
\\&&\bar \n_{Z} Y_a=e^{-u} \bar g(x, Z) a- e^{-u} \bar g(Z, a)x.\label{eq4h}
\\&&\De\bar g(x, \nu)=HV_0+ \bar g(x, \n H)-|h|^2\bar g(x, \nu), \label{eq-xnu-h}
\\&&\De V_a=n V_a-H \bar \n_\nu V_a,\label{eq-x-h}
\\&&\De \bar g(Y_a, \nu)=  -|h|^2\bar g(Y_a, \nu) +\bar g(Y_a, \n H)  +n\bar g(Y_a, \nu).\label{eq-Ynu-h}
\end{eqnarray}
\end{prop}

\begin{prop}\label{prop-3.1}Let $x: M\to\bar B^\hh_{R}$ be an  isometric immersion into the hyperbolic ball $B^\hh_{R}$ with zero mean curvature $H=0$, whose boundary $\p M$ intersects $\p B^\hh_{R}$ at a constant angle $\th \in (0, \pi)$.
For each constant vector field $a\in \rr^{n+1}$ define
\begin{eqnarray}\label{test-function}
\varphi_{a}=\frac{1}{\sin\theta\, \sinh R}\bar{g}(Y_{a}, x)+\cot\theta\,\,\bar{g}(Y_{a},\nu)\quad \text{along} \,M.
\end{eqnarray}
Then $\varphi_a$ satisfies
\begin{equation}\label{bdy-zero}
\int_{\partial M}\varphi_{a}ds=0
\end{equation}
Along $\p M$, we have
\begin{equation}\label{Hyp-bdy1}
  {\nabla}_{\mu}\varphi_{a}=q\varphi_{a}
\end{equation}
where
\begin{eqnarray}\label{qq2}
&&q=\frac{1}{\sin \th}\,\coth R+ \cot \th \, h(\mu, \mu).
\end{eqnarray}
\end{prop}

\begin{proof}
In this proof we always take value along $\p M$ and use \eqref{mu} and \eqref{nu}.
Firstly, from (\ref{KillY1}) we get
\begin{eqnarray}\label{Y-a-boundary}
\int_{M}\div_{M}(Y^{T}_{a})=\int_{M}\div_{M}(Y_{a}-\bar{g}(Y_{a},\nu)\nu)=\int_{M}\div_{M}Y_{a}-H\bar{g}(Y_{a},\nu) dA=0
\end{eqnarray}
On the other hand, using integration by parts, we see
\begin{eqnarray}\label{Y-a-boundary1}
\int_{M}\div_{M}(Y^{T}_{a})=\int_{\partial M}\bar{g}(Y_{a}^{T}, \mu)ds=\int_{\partial M}\bar{g}(Y_{a}, \mu)ds
\end{eqnarray}
Combining (\ref{Y-a-boundary}) with (\ref{Y-a-boundary1}), we get
\begin{eqnarray}\label{Y-a-boundary2}
\int_{\partial M}\bar{g}(Y_{a}, \mu)ds=0
\end{eqnarray}
Applying (\ref{mu}), (\ref{nu}) and (\ref{NN}), we have
\begin{equation}\label{test-bdy}
 \varphi_{a}=\bar{g}(Y_{a},\mu) \quad \text{on}\, \,\partial M
\end{equation}
Using (\ref{Y-a-boundary2}) and (\ref{test-bdy}), we get the first equation (\ref{bdy-zero}).\\

Next, note that
\begin{eqnarray}\label{Ya-x}
&&\bar g(Y_a, x)= e^{2u}\<Y_a, x\>=e^{2u}\frac12(1-|x|^2)\<x, a\>=e^{-u}\bar g(x, a)=V_a.
\end{eqnarray}
Therefore,
\begin{alignat}{2}\label{phi0}
\varphi_{a}=\frac{1}{\sin\theta\,\sinh R}\bar{g}(Y_{a},x )+\cot\theta\,\,\bar{g}(Y_{a},\nu)
=\frac{1}{\sin \theta\,\sinh R}V_{a}+\cot\theta\,\,\bar{g}(Y_{a},\nu).
\end{alignat}
By (\ref{eq3h}), (\ref{eq4h}) and \cite{WX} Proposition 2.1, we can compute that
\begin{alignat}{2}
\bar{\nabla}_{\mu}\varphi_{a}&=\frac{1}{\sin \theta\,\sinh R}\bar{\nabla}_{\mu}V_{a}+\cot\theta\,(\bar{g}(\bar{\nabla}_{\mu}Y_{a},\nu)+\bar{g}(Y_{a},\bar{\nabla}_{\mu}\nu))\nonumber\\
&=\frac{1}{\sin \theta \sinh R}(e^{-u}\bar{g}(\mu,a)+e^{-2u}\bar{g}(x,a)\bar{g}(x,\mu))\nonumber\\
&\quad\quad+\cot\theta\,\cdot e^{-u}[\bar{g}(x,\mu)\bar{g}(\nu,a)-\bar{g}(\mu,a)\bar{g}(x,\nu)]+\cot\theta\, h(\mu,\mu)\bar{g}(Y_{a},\mu)\nonumber.
\end{alignat}
Using $\nu=-\frac{1}{\cos \th}\bar N+\tan \th\, \mu$ and $x=\sinh R \, \bar N$, we obtain
\begin{alignat}{2}\label{phi2}
\cot\theta\,\cdot e^{-u}[\bar{g}(x,\mu)\bar{g}(\nu,a)-\bar{g}(\mu,a)\bar{g}(x,\nu)]
&=e^{-u}[\frac{1}{\sin \theta}\sinh R\cdot\bar{g}(\mu,a)-\frac{1}{\sin \theta \sinh R}\bar{g}(x,a)\bar{g}(x,\mu)]\nonumber\\
&=\frac{1}{\sin \theta \,\sinh R}(e^{-u}sinh^{2}R\cdot\bar{g}(\mu,a)-e^{-u}\bar{g}(x,a)\bar{g}(x,\mu))\nonumber
\end{alignat}
Therefore, applying relationship $\sinh^{2}R+1=\cosh^{2}R$, we have
\begin{alignat}{2}
\bar{\nabla}_{\mu}\varphi_{a}&=\frac{1}{\sin \theta \sinh R}\cosh^{2}R\cdot e^{-u}\bar{g}(\mu,a)+\frac{1}{\sin \theta \sinh R}(e^{-2u}-e^{-u})\bar{g}(x,a)\bar{g}(x,\mu)+\cot\theta\, h(\mu,\mu)\bar{g}(Y_{a},\mu)\nonumber\\
&=\frac{1}{\sin \theta \sinh R}\cosh R\,\bar{g}(\frac{1}{2}(|x|^{2}+1)a-\langle x,a\rangle x,\mu)+\cot\theta\, h(\mu,\mu)\bar{g}(Y_{a},\mu)\nonumber\\
&=\frac{1}{\sin \theta}\coth R\,\bar{g}(Y_{a},\mu)+\cot\theta\, h(\mu,\mu)\bar{g}(Y_{a},\mu)\nonumber\\
&=(\frac{1}{\sin \theta}\coth R+\cot\theta\, h(\mu,\mu))\bar{g}(Y_{a},\mu)\nonumber
\end{alignat}
By (\ref{test-bdy}), we complete this proposition.
\end{proof}
%\subsection{Uniqueness of type II stable hypersurfaces in a hyperbolic ball}
\begin{theorem}\label{thm-h} Assume $x: M\to B_R^{\hh}\subset(\bb^{n+1}, \bar g)$ is an immersed type-II stable hypersurface in the ball $B_R^{\hh}$ with zero mean curvature $H=0$ and constant contact angle $\th\in (0, \pi)$. Then $x$ is totally geodesic.
\end{theorem}
\begin{proof}
The stability inequality \eqref{second} reduces to
\begin{eqnarray}\label{stab-eq1-h}
-\int_M \varphi (\Delta \varphi+|h|^2\varphi-n\varphi)+\int_{\p M} (\n_\mu \varphi-q\, \varphi)\varphi\ge 0
\end{eqnarray}
for all function $\int_{\partial M}\varphi ds=0$, where $q$ is given by \eqref{qq2} since $\p B_R^{\hh}$ has constant principal curvature $\coth R$.\\
For each constant vector field $a\in \rr^{n+1}$, we consider
 $$\varphi_{a}=\frac{1}{\sin \theta\, \sinh R}\bar{g}(Y_{a}, x)+\cot\theta\,\bar{g}(Y_{a},\nu)$$ along $M$.

Equation (\ref{bdy-zero}) tells us that $\int_{\partial M} \varphi_a\, ds=0$. Therefore, $\varphi_a$ is an admissible function for testing stability.

Using (\ref{eq-x-h}) and (\ref{eq-Ynu-h}), noting that $H=0$ and (\ref{Ya-x}), we compute that
\begin{alignat}{2}
\Delta\varphi_{a}&=\frac{1}{\sin \theta \sinh R}\Delta V_{a}+\cot\theta\,\Delta \bar{g}(Y_{a},\nu)\nonumber\\
&=\frac{1}{\sin \theta \sinh R}(nV_{a}-H\bar{\nabla}_{\nu}V_{a})+\cot\theta\,(-|h|^{2}\bar{g}(Y_{a},\nu)+\bar{g}(Y_{a},\nabla H)+n\bar{g}(Y_{a},\nu))\nonumber\\
&=\frac{1}{\sin \theta \sinh R}nV_{a}-\cot\theta\,|h|^{2}\bar{g}(Y_{a},\nu)+n\cot\theta\, \bar{g}(Y_{a},\nu).\nonumber
\end{alignat}
Therefore, we have
\begin{alignat}{2}
  \Delta\varphi_{a}-n\varphi_{a}+|h|^{2}\varphi_{a}&=-\cot\theta\,|h|^{2}\bar{g}(Y_{a},\nu)+|h|^{2}(\frac{1}{\sin \theta \sinh R}V_{a}+\cot\theta\, \bar{g}(Y_{a},\nu))\nonumber\\
  &=\frac{1}{\sin \theta\cdot \sinh R}V_{a}|h|^{2}\nonumber
\end{alignat}
and in turn
\begin{equation}\label{inter-2}
  \varphi_{a}(\Delta\varphi_{a}-n\varphi_{a}+|h|^{2}\varphi_{a})=\frac{1}{\sin^{2}\theta \cdot \sinh^{2}R}V_{a}^{2}|h|^{2}+\frac{\cos\theta\,}{\sin^{2}\theta\cdot \sinh R}V_{a}\,\bar{g}(Y_{a},\nu)|h|^{2}.
\end{equation}
From \eqref{Hyp-bdy1}, we know
\begin{eqnarray}\label{eq-phi1-h}
\n_\mu \varphi_a-q\varphi_a=0 \quad\text{along}\,\, \partial M.
\end{eqnarray}

Inserting \eqref{inter-2} and \eqref{eq-phi1-h} into the stability condition \eqref{stab-eq1-h}, we get for any $a\in \mathbb{R}^{n+1}$,
\begin{equation}\label{sec-a}
  \frac{1}{\sin^{2}\theta\cdot \sinh^{2}R}\int_{M}(\cos\theta\,\cdot \sinh R\,\bar{g}(V_{a}Y_{a},\nu)+V_{a}^{2})|h|^{2}dA\leq0
\end{equation}

We take $a$ to be the $n+1$ coordinate vectors $\{E_i\}_{i=1}^{n+1}$ in $\R^{n+1}$.
Noticing  that $V_a=\frac{2\<x, a\>}{1-|x|^2}$, $Y_a=\frac12(|x|^2+1)a- \<x, a\>x$,
we have \begin{eqnarray*}
&&\sum_{a=1}^{n+1} V_a Y_a=x,
\\&&\sum_{a=1}^{n+1}V_a^2=\frac{4|x|^2}{(1-|x|^2)^2}=\bar g(x, x).
\end{eqnarray*}
Therefore, by summing \eqref{sec-a} for all $a$, we get
\begin{equation}\label{sum-a}
  \int_{M}(\cos\theta\,\cdot \sinh R\,\bar{g}(x,\nu)+\bar{g}(x,x))|h|^{2}dA\leq0
\end{equation}
As in the Euclidean case, we introduce an auxiliary function
\begin{equation*}
 \Phi=-(\bar{g}(x,\nu)+\cos\theta\, \,\sinh R).
\end{equation*}
From (\ref{eq-xnu-h}) and $H=0$, we get
\begin{equation}\label{Lap2}
  \Delta\Phi=\bar{g}(x,\nu)|h|^{2}
\end{equation}
Note that  $\Phi|_{\p M}=0$. Thus we have
\begin{eqnarray*}
\int_M \De \frac12\Phi^2=\int_{\p M} \Phi \n_\mu\Phi=0.
\end{eqnarray*}
Adding this to \eqref{sum-a}, using \eqref{Lap2}, we have
\begin{eqnarray*}
0&\ge&
\int_{M}(\cos\theta\,\cdot \sinh R\,\bar{g}(x,\nu)+\bar{g}(x,x))|h|^{2}+\Delta\frac12\Phi^{2}
\\&=&\int_{M}\bar{g}(x^{T}, x^{T})|h|^{2}+|\nabla\Phi|^{2}
\\&\ge &0.
\end{eqnarray*}
%The same argument as before yields the total geosesic of the immersion $x$.
This implies $x: M\to  B_R^{\hh}$ is totally geodesic. The proof is completed.
\end{proof}

\smallskip

\subsection{The spherical case}\

In this subsection, we sketch  the necessary  modifications in the case that the ambient space is the spherical space form $\ss^{n+1}$.
We use the model $$(\rr^{n+1}, \bar g_{\ss}=e^{2u} \delta) \quad \hbox{ with  }u(x)=\frac{4}{(1+|x|^2)^2},$$to represent $\ss^{n+1}\setminus\{\mathcal{S}\}$, the unit sphere without the south pole.
 Let $B_R^{\ss}$ be a ball in $\ss^{n+1}$ with radius $R\in (0, \pi)$ centered at the north pole. The corresponding $R_\rr=\sqrt{\frac{1-\cos R}{1+\cos R}}\in (0, \infty)$.
 the Killing vector field $Y_a$ in this case are
\begin{eqnarray}\label{Killsph}
Y_a=\frac12(1-|x|^2)a+\<x, a\>x.
\end{eqnarray}
The crucial functions $V_0$ and $V_a$ in this case are
$$V_0= \cos r=\frac{1-|x|^2}{1+|x|^2}, \quad V_a=\frac{2\<x, a\>}{1+|x|^2}.$$
Similarly as the hyperbolic case, these $(n+2)$ functions span the vector space  $$\{V\in C^2(\ss^{n+1}\setminus\{\mathcal{S}\}): \bar \n^2 V= - V\bar g\}.$$
Using $X_a$, $Y_a$, $V_0$ and $V_a$, the proof goes through parallel to the hyperbolic case.
The method works for balls with any radius $R\in (0, \pi)$. Compare to the hyperbolic case,  in this case $V_0=\cos r$ can be negative when $R\in (\frac{\pi}{2}, \pi)$. Nevertheless, by going through the proof, we see this does not affect the issue on stability.
\qed

\

\section{Topological restriction on type-II stable stationary surface}\label{sec4}
In this section,
we use two kinds of balancing arguments, similar to Ros-Vergasta \cite{RV} and Nunes \cite{Nu}, to get topological restriction on type-II stable stationary surfaces.

{\bf Step 1.}
Let $\hat{M}$ be a compact Riemann surface obtained from $M$ by attaching a conformal disk at any connected component of $\p M$. Then there exists a non-trivial conformal map
$\hat{\psi}:\hat{M}\to \ss^2$ such that (see \cite{GH} p. 261)
\begin{eqnarray*}
{\rm deg}(\hat{\psi})\le 1+\left[\frac{g+1}{2}\right].
\end{eqnarray*}
Let $\psi: M\to \ss^2\subset\rr^3$ be the restriction of $\hat{\psi}$ to $M$.
Then $\psi|_{\p M}: \p M\to \ss^2$ be a conformal immersion. By Hersch \cite{Hersch} and Li-Yau \cite{LiYau}'s result, by combining $\psi|_{\p M}$ with a conformal diffeomorphism of $\ss^2$, we can assume
$$\int_{\p M}\psi^i ds=0, \quad i=1,2,3,$$
where $\psi^i$ denotes the coordinate function of $\psi$ in $\rr^3$.

Using $\psi^i, i=1,2,3$ as admissible test functions in \eqref{stab-ineq0}, we get
\begin{eqnarray*}
\int_M |\n \psi^i|^2-(|h|^2+\overline{{\rm Ric}}(\nu,\nu))(\psi^i)^2 dA-\int_{\p M}  \left[\frac{1}{\sin \th}h^{\p B}(\bar \nu, \bar \nu)+\cot \th \, h(\mu,\mu)\right]  (\psi^i)^2 ds\ge 0.
\end{eqnarray*}
Summing up the above inequalities for $i=1,2,3$, using $\psi\in \ss^2$ we get
\begin{eqnarray}\label{geq1}
\int_M |\n \psi|^2-(|h|^2+\overline{{\rm Ric}}(\nu,\nu)) dA-\int_{\p M} \left[\frac{1}{\sin \th}h^{\p B}(\bar \nu, \bar \nu)+\cot \th \, h(\mu,\mu)\right] ds\ge 0.
\end{eqnarray}
By conformality of $\psi$, we have
\begin{eqnarray}\label{geq2}
\int_M |\n \psi|^2 dA< \int_{\hat M} |\n \hat{\psi}|^2 dA= 2{\rm Area}(\hat{\psi}(\hat{M}))\le 8\pi\left(1+\left[\frac{g+1}{2}\right]\right).
\end{eqnarray}

Using the Gauss equation and the Gauss-Bonnet formula, taking into account that $H=0$, we have
\begin{eqnarray*}
&&\int_{M} -|h|^2+\overline{{\rm R}}-2\overline{{\rm Ric}}(\nu,\nu) dA+2\int_{\p M}\kappa_g ds\\&=&2\int_M K dA+2\int_{\p M}\kappa_g ds=4\pi \chi(M)=4\pi(2-2g-r).
\end{eqnarray*}
Here $\kappa_g$ is the geodesic curvature of $\p M$ in $M$ and $K$ is the Gauss curvature of $M$.
It follows that
\begin{eqnarray}\label{geq3}
&&\int_{M} -(|h|^2+\overline{{\rm Ric}}(\nu,\nu)) dA=4\pi(2-2g-r)-2\int_{\p M}\kappa_g ds-\int_M (\overline{{\rm R}}-\overline{{\rm Ric}}(\nu,\nu)) dA.
\end{eqnarray}
On the other hand, from \eqref{mu} and \eqref{nu}, we have
\begin{eqnarray*}
\nu=\tan\th\,\mu-\frac{1}{\cos \th}\bar N.
\end{eqnarray*}
Thus
\begin{eqnarray*}
h(\mu,\mu)&=&H-h(e,e)=-\<\bar \n_e \nu, e\>
\\&=&-\tan\th\,\<\bar \n_e \mu, e\>+\frac{1}{\cos \th}\<\bar \n_e\bar N, e\>
\\&=&-\tan\th\,\kappa_g+\frac{1}{\cos \th}h^{\p B}(e,e),
\end{eqnarray*}
where $e\in T(\p M)$.
In turn,
\begin{eqnarray}\label{geq4}
-\int_{\p M} \left[\frac{1}{\sin \th}h^{\p B}(\bar \nu, \bar \nu)+\cot \th \, h(\mu,\mu)\right] ds=\int_{\p M} \left(-\frac{1}{\sin \th}H^{\p B}+\kappa_g\right) ds.
\end{eqnarray}
Taking into account of \eqref{geq2}, \eqref{geq3}, \eqref{geq4}  in \eqref{geq1}, we have
\begin{eqnarray}\label{geq5}
\int_M (\overline{{\rm R}}-\overline{{\rm Ric}}(\nu,\nu)) dA+\int_{\p M} \left(\kappa_g+\frac{1}{\sin \th}H^{\p B}\right) ds < 4\pi\left(4-2g-r+2\left[\frac{g+1}{2}\right]\right).
\end{eqnarray}

{\bf Step 2.} By a result of Ahlfors \cite{Ah} and Gabard \cite{Gab}, there exists a conformal branched cover $\eta: M\to \ss_+^2$ such that
\begin{eqnarray*}
{\rm deg}(\eta)\le g+r.
\end{eqnarray*}
$\eta^3=0$ on $\p M$. Also, by combining $\eta|_{\p M}$ with a conformal diffeomorphism of $\ss^1=\p \ss_+^2$, we can assume
 $$\int_{\p M}\eta^{i}ds=0, i=1,2.$$
Thus $\eta^{i}, i=1,2,3$ are admissible test functions in \eqref{stab-ineq0}. We obtain
\begin{eqnarray}\label{geq6}
\int_M |\n \eta|^2-(|h|^2+\overline{{\rm Ric}}(\nu,\nu)) dA-\int_{\p M} \left[\frac{1}{\sin \th}h^{\p B}(\bar \nu, \bar \nu)+\cot \th \, h(\mu,\mu)\right] ds\ge 0.
\end{eqnarray}
By conformality of $\eta$, we have
\begin{eqnarray}\label{geq7}
\int_M |\n \eta|^2 dA= 2{\rm Area}(\eta(M))\le 4\pi\left(g+r\right).
\end{eqnarray}
Using the same argument as Step 1, we get
\begin{eqnarray}\label{geq8}
\int_M (\overline{{\rm R}}-\overline{{\rm Ric}}(\nu,\nu)) dA+\int_{\p M} \left(\kappa_g+\frac{1}{\sin \th}H^{\p B}\right) ds\le 4\pi\left(2-g\right).
\end{eqnarray}

\noindent{\bf Proof of Theorem \ref{thm0.3}}

 Consider the free boundary case, i.e. $\th=\frac{\pi}{2}$.
In this case one has \begin{eqnarray*}
\kappa_g=h^{\p B}(e,e).
\end{eqnarray*}
Hence, from \eqref{geq5} and \eqref{geq8}, by our assumption that $\overline{{\rm Ric}}\ge 0$ and $h^{\p B}\ge 0$, we deduce
that $g=0\hbox{ or }1$ and $r=1, 2 \hbox{ or }3$, or $g=2$ and $r=1$. Moreover, $g=2$ and $r=1$ happens only when $h^{\p B}\equiv 0$ along $\p M$ and $\overline{{\rm R}}-\overline{{\rm Ric}}(\nu,\nu)\equiv 0$ along $M$.

\noindent{\bf Proof of Theorem \ref{thm0.4}}

 Let us consider the general case, $\th\in (0, \frac{\pi}{2})\cup (\frac{\pi}{2}, \pi)$. Assume $\p M$ is embedded. Then $\p M$ divide $\p B$ by several components.% \textcolor[rgb]{0.00,0.00,1.00}{ We choose one of these components, denoted by $D_{i}$, so that $M$ intersects $D_{i}$ with $\th_{i}\in (\frac{\pi}{2},\pi)$.}
% \textcolor[rgb]{0.00,0.00,1.00}{We know
%\begin{eqnarray*}
%\kappa_g=\<\bar\n_e \mu, e\>=\<\bar\n_e (\cos\th_{i}\, \bar\nu+\sin\th_{i}\,\bar N), e\>=\cos\th_{i}\, \bar \kappa_g+\sin\th_{i} \, h^{\p B}(e,e).
%\end{eqnarray*}
%where $\bar \kappa_g$ is XXXX.
%Then
%\begin{eqnarray*}
%\int_{\p M}\kappa_g ds=\sum_{i=1}^{r}\left( \cos\th_{i}\, \int_{\p D_{i}}\bar \kappa_g ds+\sin\th_{i} \, \int_{\p D_{i}}h^{\p B}(e,e)ds\right).
%\end{eqnarray*}
%By Gauss-Bonnet formula for $D_{i}$, we have
%\begin{eqnarray*}
%%\int_{D_{i}} KdA+\int_{\p D_{i}}\bar \kappa_g ds=2\pi\chi(D_{i})=2\pi.
%\end{eqnarray*}
% Hence
%\begin{eqnarray*}
%\int_{\p M}\kappa_g ds=\sum_{i=1}^{r}\left(2\pi \cos\th_{i}- \cos\th_{i} \,\int_{D_{i}} KdA+\sin\th_{i} \, \int_{\p D_{i}}h^{\p B}(e,e)ds\right).
%\end{eqnarray*}
% By the Gauss equation, if $\overline{{\rm Sect}}\ge 0$ and $h^{\p B}\ge 0$, we know the Gauss curvature $K$ of $D_{i}$ is nonnegative and in turn,
%\begin{eqnarray*}
%\int_{\p M}\kappa_g ds\ge  -2\pi r.
%\end{eqnarray*}
%Using \eqref{geq5}, we conclude that $g+\frac{r}{2}<4, \hbox{ if } g \hbox{ even }$ and $g+\frac{r}{2}<5, \hbox{ if } g \hbox{ odd}$.
%}
%%%%%%%%%%%%%%%%%%%%%%%%%%%%%%%%%%%%%%%%%%%%%%%%%%%%%%%%%%%%%%%5 
 We choose one of these components, denoted by $T$, so that $M$ intersects $T$ with $\th\in (\frac{\pi}{2},\pi)$.
We know
\begin{eqnarray*}
\kappa_g=\<\bar\n_e \mu, e\>=\<\bar\n_e (\cos\th\, \bar\nu+\sin\th\,\bar N), e\>=\cos\th\, \bar \kappa_g+\sin\th \, h^{\p B}(e,e).
\end{eqnarray*}
Then
\begin{eqnarray*}
\int_{\p M}\kappa_g ds= \cos\th\, \int_{\p M}\bar \kappa_g ds+\sin\th \, \int_{\p M}h^{\p B}(e,e)ds.
\end{eqnarray*}
By the Gauss-Bonnet formula for $T$, we have
\begin{eqnarray*}
\int_T KdA+\int_{\p T=\p M}\bar \kappa_g ds=2\pi\chi(T).
\end{eqnarray*}
Hence
\begin{eqnarray*}
\int_{\p M}\kappa_g ds=2\pi \cos\th\, \chi(T)- \cos\th \,\int_T KdA+\sin\th \, \int_{\p M}h^{\p B}(e,e)ds.
\end{eqnarray*}
Note that $\chi(T)=r \hbox{ or }2-r$ depending on the choice of $T$. By the Gauss equation, if $\overline{{\rm Sect}}\ge 0$ and $h^{\p B}\ge 0$, we know the Gauss curvature $K$ of $T$ is nonnegative and in turn,
\begin{eqnarray*}
\int_{\p M}\kappa_g ds\ge  -2\pi r.
\end{eqnarray*}
Using \eqref{geq5}, we conclude that $g+\frac{r}{2}<4, \hbox{ if } g \hbox{ even }$ and $g+\frac{r}{2}<5, \hbox{ if } g \hbox{ odd}$.

\
\section{Morse index estimate}\label{sec5}
In this section, we will discuss lower bound estimates for Morse Index of stationary type-II hypersurface in a connected domain $B$ in $\bar M^{n+1}$.

%Let $$\mathcal{S}=\left\{\vp\in C^\infty(\p M): \int_{\p M}\vp\, ds=0\right\}.$$
%Consider the following operator
%$L: \mathcal{S}\to \mathcal{S}$ given by
%$$L\vp=\n_\mu \hat{\vp}-q \hat{\vp}-\frac{1}{|\p \S|}\int_{\p \S}(\n_\mu  \hat{\vp}- q \hat{\vp}),$$
%where $ \hat{\vp}$ is the the extension of $\vp$ to $M$ by $$J\hat{\vp}:=\De \hat{\vp}+(|h|^2 +\overline{{\rm Ric}}(\nu,\nu))\hat{\vp}=0.$$
%It is easy to check that $L$ is a self-adjoint linear elliptic operator. Thus there exists a non-decreasing and diverging sequence
%$\l_1\le \l_2\le \cdots\le \l_k\uparrow\infty$ of eigenvalues of $L$ in $\mathcal{S}$ associated to a $L^2(\p M)$-orthonormal basis $\{\vp_k\}_{k=1}^\infty$ of solutions to the eigenvalue problem
%\begin{equation}\label{22}
%\begin{cases}{}
%&J\varphi=0,\\
%&\n_\mu \vp-q\vp-\frac{1}{|\p \S|}\int_{\p \S}(\n_\mu \vp-q\vp)=\l\vp.
%\end{cases}
%\end{equation}
%Let \begin{eqnarray*}
%Q(\vp, \vp)=\int_M -\vp(\De \vp+(|h|^2+\overline{{\rm Ric}}(\nu,\nu))\vp) dA+\int_{\p M} \vp(\n_\mu \vp-q \vp)\, ds.
%\end{eqnarray*}
%Set $\vp_0\equiv1$. It is standard to show that $\l_k$ can be
% variationally characterized by
%\begin{eqnarray*}
%\l_k=\inf_{\vp\in V_k^\perp\setminus\{0\}}\frac{Q(\vp, \vp)}{\int_{\p M}\vp^2\, ds}.
%\end{eqnarray*}
%where $V_k=\{\vp_0, \vp_1,\cdots, \vp_{k-1}\}$ and $V_k^\perp$ is the orthogonal complement of $V_k$ with respect to the inner product of $L^2(\p M)$.
%In this section, we estimate the Morse index for type-II stationary hypersurfaces.
For any $\varphi\in C^{\infty}(M)$, we set
\begin{eqnarray*}
Q(\vp, \vp)=\int_M -\vp(\De \vp+(|h|^2+\overline{{\rm Ric}}(\nu,\nu))\vp) dA+\int_{\p M} \vp(\n_\mu \vp-q \vp)\, ds.
\end{eqnarray*}
Here $q$ is as in \eqref{q}.
\begin{defi} Let $x: M^n\to B\subset \bar M^{n+1}$ be a type-II stationary immersed hypersurface.
The Morse index ${\rm Ind}(M)$ of $M$ is defined to be the maximal dimension of a subspace $\mathcal{S}$ of $C^\infty(M)$ such that $$\int_{\partial M} \vp\, dA=0, \hbox{ and }Q(\vp, \vp)<0,\hbox{ for }\vp\in  \mathcal{S}.$$

\end{defi}

For the Jacobi operator $J=\De+|h|^2+\overline{{\rm Ric}}(\nu,\nu)$, with the boundary condition $$\n_\mu \vp-q \vp=0,$$
there exists a non-decreasing and diverging sequence $\l_1\le \l_2\le \cdots\le \l_k\uparrow\infty$ of eigenvalues associated to a $L^2(M)$-orthonormal basis $\{\vp_k\}_{k=1}^\infty$ of solutions to the eigenvalue problem
\begin{equation}\label{22}
\begin{cases}{}
&J\varphi=-\l \vp,\hbox{ in }M\\
&\n_\mu \vp-q\vp=0, \hbox{ on }\partial M.
\end{cases}
\end{equation}
By Rayleigh's theorem, the eigenvalues $\l_{k}$ has the following variational characterization
\begin{eqnarray*}
\l_{k}=\inf_{\vp\in V_{k-1}^\perp\setminus\{0\}}\frac{Q(\vp, \vp)}{\int_{ M}\vp^2\, dA},
\end{eqnarray*}
where $V_{k-1}=\{\vp_1,\cdots, \vp_{k-1}\}$ and $V_{k-1}^\perp$ is the orthogonal complement of $V_{k-1}$ with respect to the inner product of $L^2( M)$.

The Morse index of a type-II stationary hypersurface is closely related to the number of negative eigenvalues of \eqref{22}. For type-I case, a similar statement can be found in the literature, see e.g. \cite{BB, rossman}.
\begin{prop}\label{morse-index}
Let $k$ be the number of negative eigenvalues of \eqref{22}. Then
$$ k-1\le{\rm Ind}(M)\le k.$$
\end{prop}
\begin{proof}
Let $\mathcal{V}_k=\{\vp_i\}_{i=1}^k$ be the subspace consisting of the first $k$ eigenfunctions for \eqref{22}.
Since $$Q(\vp_i, \vp_i)=\l_i\int_M \vp_i^2\, dA<0,\hbox{ for }i=1,\cdots, k,$$
we have $\mathcal{S}\subset \mathcal{V}_k$ and ${\rm Ind}(M)\le k.$

On the other hand, consider the linear operator $L: \mathcal{V}_k\to \rr$ defined by
$$L\vp=\int_{\p M} \vp\, dA.$$
It follows that $${\rm Ind}(M)=\dim{\rm Ker}(L)=\dim (\mathcal{V}_k)-\dim {\rm Im}(L)\ge k-1.$$
\end{proof}

Thanks to above proposition, to estimate the Morse index of $M$, one only needs to estimate the number of negative eigenvalues of \eqref{22}.
Next we use the method of Savo \cite{Sa}, Ambrozio-Carlotto-Sharp \cite{ACS1, ACS2} to find an estimate of number of negative eigenvalues of \eqref{22} in terms of topological invariant.

First, let $\bar M$ be isometrically embedded in $\mathbb{R}^d$. Let $\{E_i\}_{i=1}^d$ be a canonical basis of $\rr^d$. We use $\<\cdot ,\cdot\>$ to denote the inner product of $\rr^d$ and $D$ to denote the covariant derivative of $\rr^d$. We use ${\rm II}$ to denote the second fundamental form for the embedding $\bar M\subset \rr^d.$
Given a vector field $\xi$ on $M$, we denote $\xi^{b}$ by its dual 1-form.

\smallskip

\subsection{First method: coordinates of $\nu\wedge\xi$}\

Define \begin{eqnarray}\label{uij}
u_{ij}=\<\xi, E_i\>\<\nu, E_j\>-\<\xi, E_j\>\<\nu, E_i\>, \quad 1\leq i<j\leq d
\end{eqnarray}
where $\nu$ is the outward unit normal vector field of $x: M^{n}\to B\subseteq\bar M^{n+1}$. Then we have the following key formula.
\begin{prop}Let $$\xi^{b}\in \mathcal{H}^1_N(M, \p M):=\{\xi^{b}\in\Omega^{1}(M); d\xi^{b}=\delta\xi^{b}=0\,\,\text{in}\,M \,\,\text{and}\,\,\mu\wedge\xi^{b}=0 \,\,\text{on}\,\, \partial M \}.$$Then
\begin{eqnarray}\label{QQ}
&&\sum_{1\le i<j\le d} Q(u_{ij}, u_{ij})\label{sum-uij}\\
&\,\,=& \int_{M} -\sum_{k=1}^n\overline{{\rm R}}(e_k,\xi,e_k,\xi)  -\overline{{\rm Ric}}(\nu,\nu)|\xi|^2+ \sum_{k=1}^n|{\rm II}(e_k, \xi)|^2+\sum_{k=1}^n |{\rm II}(e_k, \nu)|^2|\xi|^2  \,dA\nonumber\\
&&-\int_{\p M} \frac{1}{\sin \th}\, H^{\p B}|\xi|^{2} ds\nonumber,
\end{eqnarray}
where $\{e_k\}_{k=1}^n$ is an orthonormal basis of $M$.
\end{prop}
\begin{proof} The proof is close to Savo \cite{Sa}, Ambrozio-Carlotto-Sharp \cite{ACS2}, small difference arises from the boundary computation. We prove it here for reader's convenience.
A direct computation gives
\begin{eqnarray*}
\sum_{1\le i<j\le d}|\n_{e_k} u_{ij}|^2=|D_{e_k} \xi|^2+|\xi|^2|D_{e_k}\nu|^2- 2\<D_{e_k} \nu, \xi\>^2.
\end{eqnarray*}
Note that \begin{eqnarray*}
&&D_{e_k} \xi= \n_{e_k} \xi- h(e_k, \xi)\nu+ {\rm II}(e_k, \xi),\\
&&D_{e_k}\nu=\sum_{l=1}^nh(e_k, e_l)e_l+{\rm II}(e_k, \nu).
\end{eqnarray*}
It follows that
\begin{eqnarray*}
\sum_{1\le i<j\le d}|\n_{e_k} u_{ij}|^2
&=&|\n_{e_k} \xi|^2+|h(e_k, \xi)|^2+ |{\rm II}(e_k, \xi)|^2\\&&+\sum_{l=1}^n|h(e_k, e_l)|^2|\xi|^2+ |{\rm II}(e_k, \nu)|^2|\xi|^2- 2|h(e_k, \xi)|^2,
\end{eqnarray*}
and in turn
\begin{eqnarray*}
&&\sum_{1\le i<j\le d}|\n u_{ij}|^2= \sum_{k=1}^n\sum_{1\le i<j\le d}|\n_{e_k} u_{ij}|^2
\\&=&|\n \xi|^2-\sum_{k=1}^n|h(e_k, \xi)|^2+|h|^2|\xi|^2  + \sum_{k=1}^n|{\rm II}(e_k, \xi)|^2+\sum_{k=1}^n |{\rm II}(e_k, \nu)|^2|\xi|^2.
\end{eqnarray*}
Using the Gauss equation for the immersion $x:M\to \bar M$, taking account $H=0$, we have
$$-\sum_{k=1}^n|h(e_k, \xi)|^2 = {\rm Ric}(\xi, \xi)-\sum_{k=1}^n\overline{{\rm R}}(e_k,\xi,e_k,\xi).$$
Note that $\sum_{1\le i<j\le d} |u_{ij}|^2= |\xi|^2$.
It follows that
\begin{eqnarray*}
&&\sum_{1\le i<j\le d} \int_M |\n u_{ij}|^2- (|h|^2+\overline{{\rm Ric}}(\nu,\nu))|u_{ij}|^2dA
\\&=&\int_{M} |\n \xi|^2+  {\rm Ric}(\xi, \xi)-\sum_{k=1}^n\overline{{\rm R}}(e_k,\xi,e_k,\xi) -\overline{{\rm Ric}}(\nu,\nu)|\xi|^2\\&&+ \sum_{k=1}^n|{\rm II}(e_k, \xi)|^2+\sum_{k=1}^n |{\rm II}(e_k, \nu)|^2|\xi|^2 \,dA.\end{eqnarray*}
By the Weintzenb\"ock formula, using that $d\xi^{b}=\delta \xi^{b}=0$, we have
\begin{eqnarray*}
\int_{M} |\n \xi|^2+ {\rm Ric}(\xi, \xi)\, dA
&=&\int_{\p M}g(\n_\mu \xi, \xi)+ \int_{M} -g(\n^*\n \xi, \xi)+  {\rm Ric}(\xi, \xi)\, dA
\\&=&\int_{\p M}g(\n_\mu \xi, \xi)\,ds+ \int_{M} (d\d+\d d)\xi^{b} \, dA
\\&=&\int_{\p M}g(\n_\mu \xi, \xi)\,ds.
\end{eqnarray*}
On the other hand, by choosing $\{T_{\alpha}\}_{\alpha=1}^{n-1}$ is orthonormal basis of $\partial M$, we get
$$0=\div_M \xi=\sum_{\alpha=1}^{n-1}g(\n_{T_\a} \xi, T_\a)+g(\n_{\mu} \xi, \mu).$$
Since $\xi\wedge \mu=0$ along $\p M$, then $\xi=\l \mu$ along $\p M$ for some smooth function $\l$ on $M$.
Thus\begin{eqnarray*}
&&\int_{\p M}g(\n_\mu \xi, \xi)\,ds=- \sum_{\alpha=1}^{n-1}\int_{\p M} \l g(\n_{T_\a} \xi, T_\a)\,ds=- \sum_{\alpha=1}^{n-1}\int_{\p M} \l^2 g(\n_{T_\a} \mu, T_\a)\,ds= - \int_{\p M} H^{\p M} |\xi|^2 \,ds.
\end{eqnarray*}
%Hence
%\begin{eqnarray*}
%&&\sum_{i,j=1}^d\int_M |\n u_{ij}|^2- (|h|^2+\overline{{\rm Ric}}(\nu,\nu))|u_{ij}|^2)dA
%\\&=&- \int_{\p M} H^{\p M} |\xi|^2 \,ds + \int_{M} -\sum_{k=1}^n\overline{{\rm R}}(e_k,\xi,e_k,\xi)  -\overline{{\rm Ric}}(\nu,\nu)|\xi|^2 + \sum_{k=1}^n|II(e_k, \xi)|^2+\sum_{k=1}^n |II(e_k, \nu)|^2|\xi|^2 \,dA\end{eqnarray*}
Combining the above computation, we get
\begin{eqnarray*}
\sum_{1\le i<j\le d} Q(u_{ij}, u_{ij})&=&\sum_{1\le i<j\le d}\int_M |\n u_{ij}|^2- (|h|^2+\overline{{\rm Ric}}(\nu,\nu))|u_{ij}|^2) dA\\&&-\sum_{1\le i<j\le d} \int_{\p M} \left(\frac{1}{\sin \th}\,h^{\p B}(\bar \nu, \bar \nu)+\cot \th\, h(\mu, \mu)\right)|u_{ij}|^2 ds
\\&=& \int_{M} -\sum_{k=1}^n\overline{{\rm R}}(e_k,\xi,e_k,\xi) -\overline{{\rm Ric}}(\nu,\nu)|\xi|^2+ \sum_{k=1}^n|{\rm II}(e_k, \xi)|^2+\sum_{k=1}^n |{\rm II}(e_k, \nu)|^2|\xi|^2  \,dA
\\&&-\int_{\p M}  \left(H^{\p M}+\frac{1}{\sin \th}\,h^{\p B}(\bar \nu, \bar \nu)+\cot \th\, h(\mu, \mu)\right)|\xi|^2\, ds.
\end{eqnarray*}
Recall that $$\mu=\frac{1}{\sin \th}\, \bar N+\cot \th\, \nu\quad {\rm along}\,\, \partial M$$Thus
\begin{eqnarray*}
H^{\p M}&=&\sum_{\a=1}^{n-1}  g(\n_{T_\a} \mu, T_\a)=\sum_{\a=1}^{n-1}  \bar g\left(\bar\n_{T_\a}\left(\frac{1}{\sin \th}\, \bar N+\cot \th\, \nu\right), T_\a\right)
\\&=&\frac{1}{\sin \th}\sum_{\alpha=1}^{n-1}\, h^{\p B}(T_\a, T_\a)+\cot \th\, \sum_{\alpha=1}^{n-1}h(T_\a, T_\a).
\end{eqnarray*}
It follows that
\begin{alignat}{2}
&-\int_{\p M}\left(H^{\p M}+\frac{1}{\sin \th}\,h^{\p B}(\bar \nu, \bar \nu)+\cot \th\, h(\mu, \mu)\right)|\xi|^2\, ds\nonumber\\
&= -\int_{\p M} \left(\frac{1}{\sin \th}\, H^{\p B}+\cot \th\, H\right)|\xi|^{2} ds\nonumber\\
&= -\int_{\p M} \frac{1}{\sin \th}\, H^{\p B}|\xi|^{2} ds\nonumber
\end{alignat}
where in the last equality we used $H=0$.
We get the assertion \eqref{QQ}.
%In summary, we get
%\begin{eqnarray}
%&&\sum_{i,j=1}^d Q(u_{ij}, u_{ij})\label{sum-uij}\\
%&=& \int_{M} -\sum_{k=1}^n\overline{{\rm R}}(e_k,\xi,e_k,\xi)  -\overline{{\rm Ric}}(\nu,\nu)|\xi|^2+ \sum_{k=1}^n|II(e_k, \xi)|^2+\sum_{k=1}^n |II(e_k, \nu)|^2|\xi|^2  \,dA
%\nonumber\\&&-\int_{\p M} \frac{1}{\sin \th}\, H^{\p B}|\xi|^{2} ds.\nonumber\\
%&=&\int_{M} 2\sum_{i=1}^{n}|II(e_{k},\xi)|^{2}+2\sum_{i=1}^{n}|II(e_{k},\nu)|^{2}|\xi|^{2}-(\sum_{k=1}^{n}II(e_{k},e_{k}))(II(\xi,\xi)+II(\nu,\nu)|\xi|^{2})\,dA
%\nonumber\\&&-\int_{\p M} \frac{1}{\sin \th}\, H^{\p B}|\xi|^{2} ds.\nonumber
%\end{eqnarray}
%where the last equation we use Gauss equation.
 \end{proof}

\noindent{\textbf{Proof of Theorem \ref{thm-6}}.}\
Assume the number of  negative eigenvalues of \eqref{22} is $k$. %Let $\l_1\le \l_2\le\cdots\le \l_k$ be the negative eigenvalues and $\{\vp_i\}_{i=1}^k$ be the first $k$ eigenfunctions associated to $\{\l_i\}_{i=1}^k$.
Define
\begin{eqnarray*}
\Psi: \mathcal{H}^1_N(M, \p M)&&\to \rr^{\frac{d(d-1)}{2}k}
\\ \xi^{b}&&\mapsto  \int_{ M} u_{ij}\vp_q\, dA
\end{eqnarray*}
where $u_{ij}$ are functions defined via $\xi$ by \eqref{uij} and $q$ ranges from $1$ to $k$ and $i<j$ ranges from $1$ to $d$.\\
First, we claim that ${\rm Ker}(\Psi)=\{0\}.$\\
For any non-zero 1-form $\xi^{b}\in{\rm Ker}{\Psi}$, this means that $u_{ij}\in V_{k}^\perp$, thus
\begin{equation*}
  Q(u_{ij},u_{ij})\geq\lambda_{k+1}\int_{ M}u_{ij}^{2}dA\geq0\qquad \text{for\,\,all}\,\,1\leq i<j\leq d.
\end{equation*}
by the variational characterization of the eigenvalue problem (\ref{22}). In particular,
by (\ref{sum-uij}) and hypothesis (\ref{condi-6-5}), we have
\begin{equation*}
  0\leq \sum_{1\le i<j\le d}Q(u_{ij},u_{ij})<0.
\end{equation*}
Hence $\Psi$ has trivial kernel. \\
Next, we observe that $\Psi$ is a linear operator and
 $$\mathcal{H}^1_N(M, \p M)/{\rm Ker}(\Psi)\cong {\rm Im}(\Psi)\subset \rr^{\frac{d(d-1)}{2}k}.$$
Thus $${\rm dim}(\mathcal{H}^1_N(M, \p M))\le {\rm dim}({\rm Ker}(\Psi))+\frac{d(d-1)}{2}k=\frac{d(d-1)}{2}k.$$
By \cite{ACS2} Theorem 3, we have the following isomorphisms
\begin{equation*}
\mathcal{H}_{N}^{1}(M,\partial M)\simeq H_{1}(M,\partial M;\mathbb{R}).
\end{equation*}
Recall from Proposition \ref{morse-index} that ${\rm Ind}(M)\ge k-1$.
The assertion follows.
\qed

\begin{remark}
The dimension of this homology group can be explicitly computed in terms of the homology groups of $M^{n}$ and $\partial M$(see \cite{ACS2}, Lemma 4).
\end{remark}

\smallskip

\subsection{Second method: coordinates of $\xi$}\

Now we consider ${\rm dim} M=2$ case. Assume $M^{2}$ is type-II stationary surface in $B\subseteq \bar{M}^{3}$. Given a vector field $\xi$ on $M^{2}$, we denote $\xi^{b}$ by its dual 1-form.\\
Define \begin{eqnarray}\label{ui}
u_{i}=\<\xi, E_i\>, \quad 1\leq i\leq d
\end{eqnarray}

\begin{prop}Let
\begin{equation}\label{1-form}
  \xi^{b}\in H^{1}_{T}(M,\partial M):=\{\xi^{b}\in\Omega^{1}(M):\Delta\xi^{b}=0,\,i_{\mu}\xi=0\, {\rm{and}}\,\, i_{\mu}d\xi^{b}=0\}.
\end{equation}
Then
\begin{equation}\label{1-form-sum}
\sum_{i=1}^{d} Q(u_{i}, u_{i})=\int_{M}-\frac{1}{2}\overline{\rm{R}}|\xi|^{2}+\sum_{k=1}^{2}|{\rm II}(e_{k},\xi)|^{2}dA-\int_{\partial M}\frac{1}{\sin\theta}H^{\partial B}|\xi|^{2}ds
\end{equation}
where $\{e_1, e_{2}\}$ is an orthonormal basis of $M^{2}$.
\end{prop}
\begin{proof}
We can directly compute that
\begin{equation*}
   \sum_{i=1}^{d}|\nabla u_{i}|^{2}=|D\xi|^{2}=|\nabla\xi|^{2}+\sum_{k=1}^{2}h^{2}(e_{k},\xi)+\sum_{k=1}^{2}|{\rm II}(e_{k},\xi)|^{2}
\end{equation*}
Note that
\begin{equation*}
  D_{e_{k}}\xi=\nabla_{e_{k}}\xi-h(e_{k},\xi)\nu+{\rm II}(e_{k},\xi)
\end{equation*}
Hence,
\begin{equation}\label{1-form-one}
\int_{M}\sum_{i=1}^{d}|\nabla u_{i}|^{2}\,dA=\int_{M}|\nabla\xi|^{2}+\sum_{k=1}^{2}h^{2}(e_{k},\xi)+\sum_{k=1}^{2}|{\rm II}(e_{k},\xi)|^{2}\,dA
\end{equation}
Next choosing $\{e_{1},e_{2}\}$ such that the second fundamental form $h(e_{i},e_{j})=\lambda_{i}\delta_{ij}\nu$ for $i,j=1,2$, we can check that
\begin{equation}\label{hei}
  \sum_{k=1}^{2}h^{2}(e_{k},\xi)=\lambda_{1}^{2}\langle\xi,e_{1}\rangle^{2}+\lambda_{2}^{2}\langle\xi,e_{2}\rangle^{2}=\frac{1}{2}|h|^{2}|\xi|^{2}
\end{equation}
Since $M^{2}$ is minimal surface in $B$.\\
On the other hand, applying (\ref{1-form}) and the Weintzenb\"ock formula
\begin{equation*}
  0=-\int_{M}g(\Delta\xi,\xi) dA=\int_{M}g(\nabla^{*}\nabla\xi,\xi)-Ric^{M}(\xi,\xi)\,dA
\end{equation*}
By integrating by parts and ${\rm dim} M=2$, we have
\begin{equation}\label{dw2}
  \int_{M}|\nabla\xi|^{2}+K|\xi|^{2}dA=\int_{\partial M}g(\nabla_{\mu}\xi,\xi) ds
\end{equation}
Since $i_{\mu}d\xi^{b}=0$, we know
\begin{equation}\label{dw}
  0=d\xi^{b}(\mu,\xi)=g(\nabla_{\mu}\xi,\xi)-g(\nabla_{\xi}\xi,\mu).
\end{equation}
Therefore, by (\ref{dw2}) and (\ref{dw})

\begin{eqnarray}\label{dw3}
\int_{M}|\nabla\xi|^{2}+K|\xi|^{2}dA&=&\int_{\partial M}g(\nabla_{\mu}\xi,\xi) ds=-\int_{\partial M}g(\xi,\nabla_{\xi}\mu)ds\\
&=&-\int_{\partial M}\bar{g}(\xi,\bar{\nabla}_{\xi}(\frac{1}{\sin\theta}\,\bar{N}+\cot\theta\,\nu)) ds\nonumber\\
&=&-\int_{\partial M}\frac{1}{\sin\theta}\,h^{\partial B}(\xi,\xi)+\cot\theta \, h(\xi,\xi)ds\nonumber
\end{eqnarray}
By Gauss equation and $H=0$, we have
\begin{equation}\label{Gauss}
  K=\frac{1}{2}\overline{\rm{R}}-\overline{\rm{Ric}}(\nu,\nu)-\frac{1}{2}|h|^{2}.
\end{equation}
%where $\bar{S}$ and $\overline{\rm{Ric}}$ are scalar curvature tensor and ${\rm Ricci}$ tensor of $\bar{M}$ respectively.\\
Putting (\ref{hei}),(\ref{dw3}) and (\ref{Gauss}) into (\ref{1-form-one}), we get
\begin{eqnarray}\label{dw4}
\int_{M}\sum_{i=1}^{d}|\nabla u_{i}|^{2}&=&\int_{M}(-\frac{1}{2}\overline{\rm{R}}+\overline{\rm{Ric}}(\nu,\nu)+|h|^{2})|\xi|^{2}+\sum_{k=1}^{2}|{\rm II}(e_{k},\xi)|^{2}\,dA\\
&\quad&-\int_{\partial M}\frac{1}{\sin\theta}\,h^{\partial B}(\xi,\xi)+\cot\theta \,h(\xi,\xi)\,ds\nonumber
\end{eqnarray}
Therefore, combining the above computation, we have
\begin{eqnarray}\label{dw5}
\sum_{i=1}^{d} Q(u_{i}, u_{i})&=&\sum_{i=1}^{d}\int_M |\n u_{i}|^2- (|h|^2+\overline{{\rm Ric}}(\nu,\nu))|u_{i}|^2 dA\nonumber\\
&&-\sum_{i=1}^{d} \int_{\p M} \left(\frac{1}{\sin \th}\,h^{\p B}(\bar \nu, \bar \nu)+\cot \th\, h(\mu, \mu)\right)|u_{i}|^2 ds\nonumber\\
&=&\int_{M}-\frac{1}{2}\overline{\rm{R}}|\xi|^{2}+\sum_{k=1}^{2}|{\rm II}(e_{k},\xi)|^{2}dA\nonumber\\
&&-\int_{\partial M}\frac{1}{\sin\theta}(h^{\partial B}(\xi,\xi)+h^{\partial B}(\bar{\nu},\bar{\nu})|\xi|^{2})+\cot\theta\,(h(\xi,\xi)+h(\mu,\mu)|\xi|^{2})ds\nonumber\\
&=&\int_{M}-\frac{1}{2}\overline{\rm{R}}|\xi|^{2}+\sum_{k=1}^{2}|{\rm II}(e_{k},\xi)|^{2}dA-\int_{\partial M}(\frac{1}{\sin\theta}H^{\partial B}+\cot\theta \, H)|\xi|^{2}ds\nonumber
\end{eqnarray}
Since $M^{2}$ is a minimal surface in $\bar{M}^{3}$, we complete the proof.
\end{proof}

\noindent{\textbf{Proof of Theorem \ref{thm-7-p}}.}\ Assume the number of  negative eigenvalues of \eqref{22} is $k$.
Define
\begin{eqnarray*}
\Psi: \mathcal{H}^1_T(M, \p M)&&\to \rr^{dk}
\\ \xi^{b}&&\mapsto  \int_{ M} u_{i}\vp_q\, dA
\end{eqnarray*}
where $u_{i}$ are functions defined via $\xi$ by \eqref{ui} and $q$ ranges from $1$ to $k$ and $i$ ranges from $1$ to $d$.\\
First, we claim that ${\rm Ker}(\Psi)=\{0\}.$\\
For any non-zero 1-form $\xi^{b}\in{\rm Ker}{\Psi}$, this means that $u_{i}\in V_{k}^\perp$, thus
\begin{equation*}
  Q(u_{i},u_{i})\geq\lambda_{k+1}\int_{ M}u_{i}^{2}dA\geq0\qquad \text{for\,\,all}\,\,1\leq i\leq d.
\end{equation*}
by the variational characterization of the eigenvalue problem (\ref{22}). In particular,
by (\ref{1-form-sum}) and hypothesis (\ref{condi-7}), we have
\begin{equation*}
  0\leq \sum_{i=1}^{d}Q(u_{i},u_{i})<0.
\end{equation*}
Hence $\Psi$ has trivial kernel.\\
Next, we observe that $\Psi$ is a linear operator and
 $$\mathcal{H}^1_T(M, \p M)/{\rm Ker}(\Psi)\cong {\rm Im}(\Psi)\subset \rr^{dk}.$$
Thus $${\rm dim}(\mathcal{H}^1_T(M, \p M))\le {\rm dim}({\rm Ker}(\Psi))+dk=dk.$$
By \cite{ACS2} Theorem 3 and ${\rm dim} M=2$, we have
\begin{equation*}
\mathcal{H}_{T}^{1}(M,\partial M)\simeq H_{1}(M,\partial M;\mathbb{R}).
\end{equation*}
Recall from Proposition \ref{morse-index} that ${\rm Ind}(M)\ge k-1$. We finish the proof.
\qed

\

\appendix

\section{Second variational formula}\label{appendix}
In this appendix, we prove Proposition \ref{second-var}, namely, the second variation formula of area functional (\ref{second}) under admissible wetting-area-preserving variations. For simplicity, we use $\<\cdot, \cdot\>$ to denote all the inner products in the following computation. The computation is very close to the one by Ros-Souam \cite{RS}. \\
Firstly, applying the admissible condition, (\ref{mu}) and (\ref{nu}), we get
\begin{equation*}
 0=\langle Y,\bar{N}\rangle=\langle Y, \sin\theta\,\mu-\cos\theta\,\,\nu\rangle=\sin\theta\,\langle Y,\mu\rangle-\varphi \cos\theta\,\quad \text{on}\,\,\partial M,
\end{equation*}
Therefore,
\begin{equation}\label{Y-mu}
  \langle Y,\mu\rangle=\varphi\, \cot\theta\, \quad \text{on}\,\,\partial M.
\end{equation}
So we can define
\begin{equation}\label{Y-1}
  Y=Y_{0}+\varphi\nu\triangleq Y_{1}+\cot\theta\, \varphi\mu+\varphi\nu
\end{equation}
where $Y_{1}$ denotes the tangent part of $Y$ to $\partial M$.\\
On the other hand, from $$\bar{\nu}=\cos\theta\, \mu + \sin\theta\,\nu ,$$ we see $Y$ can be also expressed as follows
\begin{equation}\label{Y-2}
  Y=Y_{1}+\frac{\varphi}{\sin\theta}(\cos\theta\,\mu+\sin\theta\nu)=Y_{1}+\frac{\varphi}{\sin\theta}\bar{\nu}.
\end{equation}

We use a prime to denote the time derivative at $t=0$ in the following.
\begin{prop} \cite{RS}\label{N'-and-v'}\
Let $\tilde{\nabla}$ denote the gradient on $\partial M$ for the metric induced by $x$ and $Y_{0}$ (resp. $Y_{1}$) the tangent part of $Y$ to $M$ (resp. to $\partial M$).Let also $S_{0}$, $S_{1}$ and $S_{2}$ denote respectively the shape operator of $M$ in $\overline{M}$ with respect to $\nu$,of $\partial M$ in $M$ with respect to $\mu$ and of $\partial M$ in $\partial B$ with respect to $\bar{\nu}$. Then
\begin{itemize}
\item[(1)] $\nu'=S_{0}(Y_{0})-\nabla \varphi.$
\item[(2)] $\mu'=-(h(Y_{0},\mu)+\nabla_{\mu}\varphi)\nu-\varphi S_{0}(\mu)+\varphi h(\mu, \mu)\mu+S_{1}(Y_{1})-\cot\theta\,\tilde{\nabla}\varphi.$
\item[(3)] $\bar{\nu}'=-h^{\partial B}(Y,\bar{\nu})\bar{N}+S_{2}(Y_{1})-\frac{1}{\sin\theta}\tilde{\nabla}\varphi.$
\end{itemize}
\end{prop}
\begin{proof}
to prove (1), let $\{e_{i}\}_{i=1}^{n}$ be an orthonormal basis of $T_{p}M$ for some $p\in M$. Put $e_{i}(t)=(x(t,\cdot))_{\ast}(e_{i})$, then using the fact $\langle e_{i}(t),\nu(t)\rangle=0$ and $[e_{i}(t), Y(t)]=0$, we have
\begin{alignat}{2}
\nu'&=\sum_{i=1}^{n}\langle \nu', e_{i}\rangle e_{i}=-\sum_{i=1}^{n}\langle\nu, e_{i}'\rangle e_{i}\nonumber\\
&=-\sum_{i=1}^{n}\langle\nu, \bar \n_{e_{i}}Y\rangle e_{i}=-\sum_{i=1}^{n}\langle\nu,\bar \n_{e_{i}}(Y_{0}+\varphi\nu) \rangle e_{i}\nonumber\\
&=\sum_{i=1}^{n}\langle S_{0}(Y_{0}), e_{i}\rangle e_{i}-\sum_{i=1}^{n}d\varphi(e_{i})e_{i}\nonumber\\
&=S_{0}(Y_{0})-\nabla \varphi\nonumber
\end{alignat}
As a consequence of (1) we get
\begin{equation}\label{nu-n}
  \langle\mu',\nu\rangle=-\langle\mu,\nu'\rangle=-h(Y_{0}, \mu)+\nabla_{\mu}\varphi.
\end{equation}
Let now $\{T_\a\}_{\a=1}^{n-1}$ be an orthonormal basis of $T_{p}(\partial M)$ for some $p\in \partial M$. As before, put $T_\a(t)=(x(t,\cdot))_{\ast}(T_\a)$, then we can use (\ref{Y-1}) and $[T_\a(t), Y(t)]=0$
\begin{alignat}{2}
\langle\mu', T_\a\rangle&=-\langle\mu,T'_{\a}\rangle=-\langle\mu,\bar \n_{T_\a}Y\rangle=-\langle\mu, \bar \n_{T_\a}(Y_{1}+\cot\theta\, \varphi\mu+\varphi\nu)\rangle\label{nu-n2}\\
&=-\langle\mu, \bar \n_{T_\a}Y_{1}\rangle-\cot\theta\, d\varphi(T_\a)-\varphi\langle\mu, \bar \n_{T_\a}\nu\rangle\nonumber
\end{alignat}
Thanks to (\ref{nu-n}) and (\ref{nu-n2}), we have
\begin{alignat}{2}
\mu'&=\langle\mu',\mu\rangle\mu+\langle\mu', \nu\rangle\nu+\sum_{\a=1}^{n-1}\langle\mu', T_\a\rangle T_\a\label{second-term}\\
&=(-h(Y_{0}, \mu)+\nabla_{\mu}\varphi)\nu-\sum_{\a=1}^{n-1}\langle\mu, \bar \n_{T_\a}Y_{1}\rangle T_\a-\cot\theta\,\sum_{\a=1}^{n-1}d\varphi(T_\a)T_\a-\varphi\sum_{\a=1}^{n-1}\langle S_{0}(\mu), T_\a\rangle T_\a\nonumber\\
&=(-h(Y_{0}, \mu)+\nabla_{\mu}\varphi)\nu+S_{1}(Y_{1})-\cot\theta\,\tilde{\nabla}\varphi-\varphi(S_{0}(\mu)-h(\mu,\mu)\mu)\nonumber
\end{alignat}
The formula (2) follows from (\ref{second-term}).\\
To prove (3), we use $[T_\a(t), Y(t)]=0$ again and (\ref{Y-2})
\begin{alignat}{2}
\langle\bar{\nu}',T_\a\rangle&=-\langle\bar{\nu},T'_{\a}\rangle=-\langle\bar{\nu},\bar \n_{T_\a}Y\rangle\label{third-term}\\
&=-\langle\bar{\nu},\bar \n_{T_\a}Y_{1}\rangle-\frac{1}{\sin\theta}\,d\varphi(T_\a)\nonumber\\
&=\langle S_{2}(Y_{1}),T_\a\rangle-\frac{1}{\sin\theta}\,d\varphi(T_\a)\nonumber
\end{alignat}
Therefore, the formula (3) now follows from (\ref{third-term}) and the fact $\langle\bar{\nu}',\bar{N}\rangle=-h^{\partial B}(Y,\bar{\nu}).$
\end{proof}

\begin{prop}\label{N-and-n}\
\begin{equation*}
\langle S_{1}(Y_{1}),Y_{1}\rangle-\cot\theta\,\langle\tilde{\nabla}\varphi,Y_{1}\rangle=\cos\theta\,\,\langle Y, \bar{\nu}'\rangle+\sin\theta \,h^{\partial B}(Y_{1},Y_{1})\quad \text{along} \quad\partial M.\quad
\end{equation*}
\end{prop}
\begin{proof}
By Proposition \ref{N'-and-v'} (3) and the fact $\<Y,\bar N\>=0$, we have
\begin{alignat}{2}
\cos\theta\,\,\langle Y,\bar{\nu}'\rangle&=\cos\theta\,\langle Y,-h^{\partial B}(Y,\bar{\nu})\bar{N}+S_{2}(Y_{1})-\frac{1}{\sin\theta}\tilde{\nabla}\varphi\rangle\\
&=\langle D_{Y_{1}}(\cos\theta\,\bar{\nu}), Y_{1}\rangle-\cot\theta\,\langle\tilde{\nabla}\varphi,Y_{1}\rangle\nonumber\\
&=\langle D_{Y_{1}}(\mu-\sin\theta\bar{N}),Y_{1}\rangle-\cot\theta\,\langle\tilde{\nabla}\varphi,Y_{1}\rangle\nonumber\\
&=\langle S_{1}(Y_{1}),Y_{1}\rangle-\sin\theta\,h^{\partial B}(Y_{1},Y_{1})-\cot\theta\,\langle\tilde{\nabla}\varphi,Y_{1}\rangle.\nonumber
\end{alignat}
\end{proof}

Now we are ready to prove Proposition \ref{second-var}.

\noindent{\bf Proof of Proposition \ref{second-var}.}
Firstly, from (\ref{first-var}), we have
\begin{alignat*}{2}
A''(0)=\int_{M}&H'\varphi+H\vp'\, dA+\int_{\partial M}\langle Y', \mu\rangle+\langle Y, \mu'\rangle ds+\int_{\partial M}\langle Y,\mu\rangle\frac{d}{dt}\Big|_{t=0}ds_{t}.
\end{alignat*}

Notice that
$A'(0)=0$ for any wetting-area-preserving variations if and only if $H=0$ in $M$ and $\theta$ is constant. Moreover, we have the well known formula (see \cite{H.R})
\begin{equation*}
  H'=-(\Delta \vp+|h|^{2}\vp+\overline{\rm Ric}(\nu,\nu)\vp).
\end{equation*}
It follows that
\begin{alignat}{2}\label{variation-formula}
A''(0)=-\int_{M}\varphi(\Delta \vp+|h|^{2}\vp+\overline{\rm Ric}(\nu,\nu)\vp)\, dA+\int_{\partial M}\langle Y', \mu\rangle+\langle Y, \mu'\rangle ds+\int_{\partial M}\langle Y,\mu\rangle\frac{d}{dt}\Big|_{t=0}ds_{t}.
\end{alignat}
So to prove the formula for $A''(0)$ we need to compute
\begin{equation}\label{boundary-term}
\int_{\partial M}\langle Y', \mu\rangle+\langle Y, \mu'\rangle ds+\int_{\partial M}\langle Y,\mu\rangle\frac{d}{dt}\Big|_{t=0}ds_{t}.
\end{equation}
Firstly, we compute the first term of (\ref{boundary-term}).
By using (\ref{mu}), we have
\begin{alignat}{2}
\langle Y',\mu\rangle&=\langle Y',\sin\theta\, \bar{N}+\cos\theta\, \bar{\nu}\rangle\label{Y'}\\
&=\sin\theta\, \langle D_{Y}Y,\bar{N}\rangle+\cos\theta\,\langle Y',\bar{\nu}\rangle\nonumber\\
&=-\sin\theta \,  h^{\partial B}(Y,Y)+\cos\theta\,\langle Y',\bar{\nu}\rangle.\nonumber
\end{alignat}
By (\ref{Y-2}), we obtain that
\begin{equation}\label{pi}
   h^{\partial B}(Y,Y)= h^{\partial B}(Y_{1},Y_{1})+\frac{2\varphi}{\sin\theta}\, h^{\partial B}(Y_{1},\bar{\nu})+\frac{\varphi^{2}}{\sin^{2}\theta}\,  h^{\partial B}(\bar{\nu},\bar{\nu}).
\end{equation}
Applying (\ref{Y'}) and (\ref{pi}), we have
\begin{alignat}{2}
\int_{\partial M}\langle Y', \mu\rangle ds&=-\sin\theta\int_{\partial M}h^{\partial B}(Y,Y)ds+\cos\theta\,\int_{\partial M}\langle Y',\bar{\nu}\rangle ds\label{1.5-term}\\
&=-\sin\theta\, \int_{\partial M}h^{\partial B}(Y_{1},Y_{1})+\frac{2\varphi}{\sin\theta} \, h^{\partial B}(Y_{1},\bar{\nu})+\frac{\varphi^{2}}{\sin^{2}\theta}\,  h^{\partial B}(\bar{\nu},\bar{\nu})ds+\cos\theta\,\int_{\partial M}\langle Y',\bar{\nu}\rangle ds.\nonumber
\end{alignat}
On the other hand, using (\ref{mu}) and (\ref{nu}), we get
\begin{alignat}{2}
h(Y_{1},\mu)=\langle\mu,\bar\n_{Y_{1}}\nu\rangle&=\langle \sin\theta\bar{N}+\cos\theta\,\bar{\nu}, \bar\n_{Y_{1}}(-\cos\theta\,\bar{N}+\sin\theta\bar{\nu})\rangle\label{first-term1}\\
&=-\langle\bar{\nu}, \bar\n_{Y_{1}}\bar{N}\rangle=-h^{\partial B}(Y_{1}, \bar{\nu})\nonumber
\end{alignat}
Therefore, inserting (\ref{first-term1}) into (\ref{1.5-term}), we get the first term of (\ref{boundary-term})
\begin{alignat}{2}
\int_{\partial M}&\langle Y', \mu\rangle ds=-\sin\theta\int_{\partial M}h^{\partial B}(Y,Y)\,ds+\cos\theta\,\int_{\partial M}\langle Y',\bar{\nu}\rangle \,ds\label{first-term}\\
&=-\sin\theta\int_{\partial M}h^{\partial B}(Y_{1},Y_{1})-\frac{2\varphi}{\sin\theta} \,h(Y_{1},\mu)+\frac{\varphi^{2}}{\sin^{2}\theta}  \,h^{\partial B}(\bar{\nu},\bar{\nu})ds+\cos\theta\int_{\partial M}\langle Y',\bar{\nu}\rangle  \,ds\nonumber
\end{alignat}

\noindent By Proposition \ref{N'-and-v'} (2), Proposition \ref{N-and-n} and (\ref{Y-1}), we can compute the second term of (\ref{boundary-term}) as follow
\begin{alignat}{2}
&\int_{\partial M}\langle Y, \mu'\rangle\, ds\label{second-terms}\\
&=\int_{\partial M}(-h(Y_{0},\mu)+\nabla_{\mu}\varphi)\langle Y,\nu\rangle-\varphi\langle S_{0}(\mu), Y\rangle+\varphi h(\mu,\mu)\langle Y,\mu\rangle+\langle S_{1}(Y_{1}),Y\rangle-\cot\theta\,\langle\tilde{\nabla}\varphi, Y\rangle \,ds\nonumber\\
&=\int_{\partial M}\varphi\nabla_{\mu}\varphi-2h(Y_{0},\mu)\varphi+\varphi h(\mu,\mu)\langle Y_{1}+\cot\theta\, \varphi\mu+\varphi\nu,\mu\rangle+\langle S_{1}(Y_{1}),Y\rangle-\cot\theta\,\langle\tilde{\nabla}\varphi, Y\rangle\, ds\nonumber\\
&=\int_{\partial M}\varphi\nabla_{\mu}\varphi-2h(Y_{0},\mu)\varphi+\cot\theta\, \varphi^{2}h(\mu,\mu)+\langle S_{1}(Y_{1}),Y\rangle-\cot\theta\,\langle\tilde{\nabla}\varphi, Y\rangle \,ds\nonumber\\
&=\int_{\partial M}\varphi\nabla_{\mu}\varphi-2h(Y_{1}+\cot\theta\, \varphi\mu,\mu)\varphi+\cot\theta\, \varphi^{2}h(\mu,\mu)+\cos\theta\,\langle Y,\bar{\nu}'\rangle+\sin\theta\,h^{\partial B}(Y_{1},Y_{1})\,ds\nonumber\\
&=\int_{\partial M}\varphi\nabla_{\mu}\varphi-2h(Y_{1}, \mu)\varphi-\cot\theta\, \varphi^{2}h(\mu,\mu)+\cos\theta\,\langle Y,\bar{\nu}'\rangle+\sin\theta\,h^{\partial B}(Y_{1},Y_{1})\,ds.\nonumber
\end{alignat}

Next, we compute the third boundary term of (\ref{boundary-term}) by using (\ref{mu})
\begin{equation}\label{third-terms}
\int_{\partial M}\langle Y,\mu\rangle\frac{d}{dt}\Big|_{t=0}ds_{t}=\int_{\partial M}\langle Y, \sin\theta \bar{N}+\cos\theta\,\bar{\nu}\rangle\frac{d}{dt}\Big|_{t=0}ds_{t}=\cos\theta\,\int_{\partial M}\langle Y,\bar{\nu}\rangle\frac{d}{dt}\Big|_{t=0}ds_{t}.
\end{equation}

Therefore, putting (\ref{first-term}), (\ref{second-terms}), (\ref{third-terms}) into (\ref{variation-formula}), we get
\begin{alignat}{2}
 A''(0)&=-\int_{M}\varphi(\Delta \varphi+(|h|^{2}+\overline{\rm Ric}(\nu,\nu))\varphi)dA+\int_{\partial M}\varphi(\nabla_{\mu}\varphi-q\varphi)ds+\cos\theta\,\frac{d}{dt}\Big|_{t=0}\left(\int_{\partial M}\langle Y, \bar{\nu}\rangle ds_{t}\right)\nonumber\label{variation-formula}\\
&=-\int_{M}\varphi(\Delta \varphi+(|h|^{2}+\overline{\rm Ric}(\nu,\nu))\varphi)dA+\int_{\partial M}\varphi(\nabla_{\mu}\varphi-q\varphi)ds\nonumber
\end{alignat}
where in the last equality we have used the wetting-area-preserving condition
$$A_W'(0)=\frac{d}{dt}\Big|_{t=0}\left(\int_{\partial M}\langle Y, \bar{\nu}\rangle ds_{t}\right)=0$$
and the expression \eqref{q} for $q$.
%\begin{equation*}
%q=\frac{1}{\sin\theta}\,h^{\partial B}(\bar{\nu}, \bar{\nu})+\cot\theta\, h(\mu, \mu)
%\end{equation*}
The proof is completed.
\qed

\

\end{document}